\documentclass[11pt,reqno]{amsart}

\usepackage[a4paper,margin=26mm]{geometry}
\usepackage{amsmath,amssymb,amsthm,mathtools}
\usepackage[T1]{fontenc}
\usepackage{lmodern}
\usepackage{microtype}
\usepackage{xcolor}
\usepackage[colorlinks=true,linkcolor=blue!45!black,citecolor=blue!45!black,urlcolor=blue!45!black]{hyperref}
\hypersetup{
  pdftitle={Mod 3 congruence families for twisted signatures of string manifolds in dimensions up to twenty-eight},
  pdfauthor={Tong WU and Yong Wang},
  pdfsubject={Modular forms, twisted signatures, and Dirac-index congruences},
  pdfkeywords={string manifold, twisted signature, Dirac index, modular form, congruence}
}

\allowdisplaybreaks[3]
\numberwithin{equation}{section}

\newtheorem{theorem}{Theorem}[section]
\newtheorem{proposition}[theorem]{Proposition}
\newtheorem{lemma}[theorem]{Lemma}
\newtheorem{corollary}[theorem]{Corollary}

\theoremstyle{definition}
\newtheorem{definition}[theorem]{Definition}

\theoremstyle{remark}
\newtheorem{remark}[theorem]{Remark}

\newtheorem{bundlecor}{Corollary}[section]
\theoremstyle{definition}
\newtheorem{bundleexample}{Example}[section]
\theoremstyle{remark}

\newcommand{\Ahat}{\widehat{A}}
\newcommand{\Lhat}{\widehat{L}}
\newcommand{\ch}{\operatorname{ch}}
\newcommand{\Ind}{\operatorname{Ind}}
\newcommand{\Sig}{\operatorname{Sig}}

\newcommand{\T}{T_{\mathbb C}M}

\newcommand{\Z}{\mathbb Z}

\newcommand{\cA}{\mathcal A}
\newcommand{\cB}{\mathcal B}
\newcommand{\cI}{\mathcal I}
\newcommand{\cS}{\mathcal S}

\title[Mod $3$ congruence families]{Mod $3$ congruence families for twisted signatures\\
of string manifolds in dimensions up to twenty-eight}

\author{Tong Wu}
\address{Department of Mathematics, Northeastern University, Shenyang, 110819, China}
\email{wut977@nenu.edu.cn}

\author{Yong Wang}
\address{School of Mathematics and Statistics, Northeast Normal University, Changchun 130024, China}
\email{wangy581@nenu.edu.cn}


\begin{document}

\begin{abstract}
In the paper, we establish congruences for twisted Dirac indices and signatures on
closed string manifolds by specializing the two-parameter modular forms
of Liu and Wang to the tangent bundle. Fourier coefficient comparisons
yield uniform congruences modulo $3$ in dimensions $4k$,
$1\leq k\leq7$. Suitable parameter choices give explicit formulas
involving only genuine twisting bundles in dimensions
$8,12,16,20,24,$ and $28$, and recover the comparisons of Chen and Han
in dimension $24$. The level-one Witten genus further strengthens
selected Dirac congruences to moduli $729,27,27,$ and $9$ in dimensions
$16,20,24,$ and $28$, respectively. The proofs use Fourier expansions
in reduced $K$-theory and integer identities obtained by clearing
powers of two from the denominators before reduction modulo $3$.
We also construct non-string spin counterexamples in each of the six
listed dimensions and identify the additional coefficient that
obstructs extending the uniform argument to dimension $32$ without
further information.
\end{abstract}

\maketitle

\section{Introduction}

The index of an elliptic operator is an integer expressed by a rational
characteristic number. This relation underlies many divisibility and
congruence theorems.  The signature theorem of Hirzebruch
\cite{Hirzebruch} and the index theorem of Atiyah and Singer
\cite{AtiyahSingerIII} are the classical examples.  On spin manifolds, the
Rokhlin--Ochanine divisibility theorem gives further restrictions on the
signature; see \cite{Ochanine,Rokhlin}.  Elliptic genera add modular symmetry
to this picture.  Witten's construction \cite{WittenElliptic} and its
mathematical developments by Landweber--Stong and Liu
\cite{LandweberStong,LiuCMP,LiuTopology} show that a modular transformation
can convert an equality of characteristic forms into a family of arithmetic
identities between indices.

The present paper concerns the generalized anomaly-cancellation forms of
Liu and Wang \cite{LiuWang}.  Their construction depends on two integers
$(a,b)$ and an auxiliary even-rank Euclidean bundle $V$.  Chen and Han
\cite{ChenHan} applied the specialization
\[
             \dim M=24,\qquad V=TM,\qquad (a,b)=(0,1)
\]
to obtain a mod $3$ congruence between a twisted signature and a twisted
Dirac index.  Together with the $(1,0)$ comparison and a second-coefficient identity
for the Witten genus, they proved
\[
 \Sig(M,\Lambda^2T_{\mathbb C}M)
 \equiv \Ind(D\otimes T_{\mathbb C}M)\pmod3
\]
for every $24$-dimensional string manifold.  In particular, the twisted
signature on the left is divisible by $3$ because the Rarita--Schwinger
index on the right is divisible by $3$ in this dimension
\cite[Theorem 0.1]{ChenHan}.

We retain both parameters and compute the relevant Fourier coefficients
uniformly.  Working with the reduced tangent bundle separates these
coefficients from dimension-dependent trivial summands.  A triangular
expansion in the level-two modular generators then gives comparisons in
dimensions $4,8,12,16,20,24,28$.  The parameter analysis below concerns
specific cancellations of universal coefficients; it makes no claim that
other parameter values cannot yield additional divisibility on particular
manifolds.

\subsection*{Notation for the index pairings}
Let $M^{4k}$ be closed and string.  In particular, $M$ is spin, so its
complex spinor bundle $\cS(M)=\Delta(TM)$ is defined.  For a virtual complex
bundle $F$ and an integer $r\geq0$, set
\begin{equation}\label{eq:index-family-intro}
 \cI_r(F):=\left\langle
       \Ahat(TM)\ch\!\left(\cS(M)^{\otimes r}\right)\ch(F),[M]
                  \right\rangle.
\end{equation}
Thus $\cI_0(F)$ is the index of the spin Dirac operator twisted by $F$.
The identity $\Lhat(TM)=\Ahat(TM)\ch(\cS(M))$ gives
\begin{equation}\label{eq:I-as-signature-intro}
 \cI_r(F)=\Sig\!\left(M,\cS(M)^{\otimes(r-1)}\otimes F\right),
 \qquad r\geq1.
\end{equation}
Both identities extend additively to virtual bundles.  Complete definitions
and normalization conventions are given in Section~\ref{sec:prelim}.

Put $E=\widetilde{T_{\mathbb C}M}=T_{\mathbb C}M-\mathbb C^{4k}$.  For
$a,b\geq0$, define the following three virtual bundles:
\begin{align}
 \cA_{a,b}&=(1+a)E+b(\Lambda^2E-S^2E),\label{eq:A-intro}\\
 \cB^{(1)}_{a,b}&=(b-a)E,\label{eq:B1-intro}\\
 \cB^{(2)}_{a,b}&=(1+b)E+(a+b)\Lambda^2E
 +\frac{(a-b)^2-(a+b)}2 E\otimes E.\label{eq:B2-intro}
\end{align}
The coefficient in the last line is integral because
$(a-b)^2\equiv a+b\pmod2$.  These bundles arise
as the $q$ and $q^{1/2},q$ coefficients of the two Liu--Wang theta bundles;
see Proposition~\ref{prop:theta-coefficients}.

Our first main result is a uniform comparison across dimensions.
\begin{theorem}\label{thm:main-intro}
Let $M^{4k}$ be a closed string manifold for $1\le k\le 7$, and let $a,b\in \mathbb Z_{\ge 0}$. With the bundles defined in \eqref{eq:A-intro}--\eqref{eq:B2-intro}, we have the following congruences:
\begin{align}
 k=1:&\qquad \cI_a(\cA_{a,b})\equiv0\pmod3;\label{eq:main-k1}\\
 k=2,3:&\qquad \cI_a(\cA_{a,b})+
             \cI_b(\cB^{(1)}_{a,b})\equiv0\pmod3;\label{eq:main-k23}\\
 k=4,5,6,7:&\qquad \cI_a(\cA_{a,b})-
             \cI_b(\cB^{(2)}_{a,b})\equiv0\pmod3.\label{eq:main-k456}
\end{align}
\end{theorem}
\begin{bundlecor}[Genuine-bundle form of Theorem~\ref{thm:main-intro}]
\label{cor:genuine-uniform}
Under the hypotheses of Theorem~\ref{thm:main-intro}, put $d=4k$ and
$T=T_{\mathbb C}M$.  Introduce the integer indices
\[
 J_r=\cI_r(\mathbb C),\quad U_r=\cI_r(T),\quad
 V_r=\cI_r(\Lambda^2T),\quad W_r=\cI_r(S^2T),\quad
 Z_r=\cI_r(T\otimes T).
\]
Every twisting bundle in these definitions is genuine; in particular,
$J_r=\Ind(D\otimes\cS(M)^{\otimes r})$, and similarly for the
other four indices.  For $r\geq1$, these may equivalently be written as
twisted signatures using \eqref{eq:I-as-signature-intro}.  Set
\begin{align*}
 c&=\frac{(a-b)^2-(a+b)}2,\\
 \alpha&=1+b-d(a+b)-2dc,\\
 \beta&=-d(1+b)+(a+b)\binom{d+1}{2}+cd^2.
\end{align*}
Then the three congruences of that theorem become, respectively:
for $k=1$,
\begin{equation*}
 (1+a)U_a+bV_a+d(b-a-1)J_a
       \equiv bW_a\pmod3;\tag{C1.1}\label{eq:genuine-uniform-1}
\end{equation*}
for $k=2,3$,
\begin{equation*}
\begin{split}
 &(1+a)U_a+bV_a+d(b-a-1)J_a+(b-a)U_b\\
 &\hspace{25mm}\equiv bW_a+d(b-a)J_b\pmod3;
\end{split}\tag{C1.2}\label{eq:genuine-uniform-23}
\end{equation*}
for $k=4,5,6,7$,
\begin{equation*}
\begin{split}
 &(1+a)U_a+bV_a+d(b-a-1)J_a\\
 &\hspace{12mm}\equiv bW_a+(a+b)V_b+cZ_b+\alpha U_b+\beta J_b\pmod3.
\end{split}\tag{C1.3}\label{eq:genuine-uniform-high}
\end{equation*}
Thus every index in these congruences is twisted by a genuine bundle.
\end{bundlecor}

Corollary~\ref{cor:genuine-uniform} is the genuine-bundle reformulation of
Theorem~\ref{thm:main-intro}; it has the same hypotheses and modulus.
The next theorem records the concise genuine-bundle congruences obtained
from selected parameter choices, using the additional coefficient
relations established in Sections~\ref{sec:families} and~\ref{sec:dimensions}.

\begin{theorem}
\label{thm:selected-intro}\label{cor:genuine-selected}
Let $M^d$ be a closed string manifold, where $d=8,12,16,20,24$, or $28$.
Use the genuine-bundle indices $J_r,U_r,V_r,W_r$ of
Corollary~\ref{cor:genuine-uniform}.  The following two congruences hold
in each indicated dimension.  Every congruence in the table is modulo $3$.
\[
\renewcommand{\arraystretch}{1.5}
\begin{array}{c|c|c}
 d&\text{Dirac congruence}&\text{spinor-twisted congruence}\\ \hline
 8&U_0\equiv2J_0&U_2+V_0+2J_0\equiv2J_2+W_0\\
 12&U_0\equiv0&U_2+V_0\equiv W_0\\
 16,28&W_0\equiv J_0&W_2+V_0+J_0\equiv U_2+U_0\\
 20&W_0\equiv U_0+J_0&W_2+V_0+J_2\equiv2(U_2+U_0)\\
 24&W_0+U_0\equiv0&W_2+V_0\equiv0
\end{array}\tag{C2}\label{eq:genuine-selected-table}
\]
In particular, the two twenty-four-dimensional formulas are
\begin{align*}
 \Ind(D\otimes S^2T)+\Ind(D\otimes T)&\equiv0\pmod3,\\
 \Sig(M,\cS(M)\otimes S^2T)+\Ind(D\otimes\Lambda^2T)
       &\equiv0\pmod3.
\end{align*}
All twists appearing here are genuine complex vector bundles.
\end{theorem}

Theorem~\ref{thm:selected-intro} gives selected mod $3$ consequences,
whereas the following theorem strengthens its Dirac column in dimensions
$16,20,24,28$ under the same string hypothesis.  The extra input is the
level-one Witten genus.  No higher modulus is asserted for the
spinor-twisted column or for the full $(a,b)$-comparison in
Theorem~\ref{thm:main-intro}, and the displayed moduli are not claimed
to be optimal.  The proof is given in Section~\ref{sec:higher-three}.

\begin{theorem}\label{thm:higher-three}
Let $M^d$ be a closed string manifold, where $d=16,20,24$, or $28$, and
put $T=T_{\mathbb C}M$.  With the genuine-bundle indices
\[
 J_0=\Ind(D)=\Ahat(M),\qquad U_0=\Ind(D\otimes T),\qquad
 W_0=\Ind(D\otimes S^2T),
\]
the following congruences hold:
\begin{equation}\label{eq:higher-three-table}
\renewcommand{\arraystretch}{1.35}
\begin{array}{c|c|c}
 d&\text{congruence}&\text{modulus}\\ \hline
 16&W_0\equiv J_0&729\\
 20&W_0\equiv0&27\\
 24&W_0+U_0\equiv0&27\\
 28&W_0\equiv J_0&9
\end{array}.
\end{equation}
In dimension $20$, one also has $U_0+J_0\equiv0\pmod{27}$, so that
$W_0\equiv U_0+J_0\pmod{27}$.
In particular, all four Dirac congruences in dimensions $16,20,24,28$
in Theorem~\ref{thm:selected-intro} remain valid with modulus $9$
in place of $3$.
\end{theorem}

The two parameter choices used for the high-dimensional statements are
$(2,0)$ and $(0,2)$.  They must be distinguished from the diagonal
choice $(2,2)$: for every $s\geq0$,
\[
 \cB^{(2)}_{s,s}=(1+s)E+2s\Lambda^2E-sE^{\otimes2}
               =\cA_{s,s}.
\]
Consequently the high-dimensional comparison with $a=b=s$ is an
identity and supplies no additional congruence by itself.

This paper is organized as follows.
Section~\ref{sec:prelim} introduces the geometric conventions and
modular forms used throughout the paper.
In Section~\ref{sec:theta-bundles}, we establish the specialized
modular transformation law and compute the Fourier coefficients
in reduced $K$-theory.
These calculations lead to the proof of
Theorem~\ref{thm:main-intro} in Section~\ref{sec:comparison},
where denominators are cleared before taking residues.
Section~\ref{sec:families} studies the three parameter families in
dimension $24$ and recovers the Chen--Han comparisons.
Together with the remaining dimension calculations in
Section~\ref{sec:dimensions}, this proves
Theorem~\ref{thm:selected-intro}.
The latter section also discusses the limitation in dimension $32$
and proves Theorem~\ref{thm:higher-three} using the level-one
Witten genus.
Finally, in Section~\ref{sec:examples}, we present two non-string
counterexamples illustrating the role of the string hypothesis.

\section{Preliminaries}\label{sec:prelim}

We recall the geometric and modular conventions used below.
\subsection{Spin and string structures}

All manifolds in this paper are smooth, closed, and oriented. Let $M$ be of real dimension $4k$, with $TM$ equipped with a Riemannian metric and its Levi--Civita connection $\nabla^{TM}$. The orientation then gives rise to an $\mathrm{SO}(4k)$-principal bundle of oriented orthonormal frames.

\begin{definition}\cite[Chapter II]{LawsonMichelsohn} and \cite{StolzTeichner}\label{def:spin-string}
A \emph{spin structure} on $M$ is a lift of the oriented frame bundle from
$SO(4k)$ to $\operatorname{Spin}(4k)$.  Such a lift exists precisely when
$w_2(TM)=0$.  For a spin bundle, the class $p_1(TM)$ has a canonical half
$\tfrac12p_1(TM)\in H^4(M;\mathbb Z)$.  A \emph{string structure} is a spin
structure together with a trivialization of this half Pontryagin class.  In
particular,
\begin{equation}\label{eq:string-p1}
             \frac12p_1(TM)=0\quad\Longrightarrow\quad p_1(TM)=0.
\end{equation}
\end{definition}

Only the cohomological consequence in \eqref{eq:string-p1} enters the
characteristic-number calculations.  The stronger language of a string
structure is retained because it is the natural hypothesis for the Witten
genus; the underlying spin structure supplies the spinor bundle used
below.  In fact, the characteristic-number arguments only require that
$M$ be spin and that $p_1(TM)$ vanish in rational cohomology.  This weaker
condition need not imply the existence of a string structure.  Background on
the obstruction-theoretic definition can be found in
\cite[Chapter II]{LawsonMichelsohn} and \cite{StolzTeichner}.

Let $\cS(M)=\Delta(TM)$ be the complex spinor bundle associated with the
complex spin representation.  Since the dimension is even, it is
$\mathbb Z_2$-graded,
\[
                 \cS(M)=\cS^+(M)\oplus\cS^-(M).
\]
We use $\cS(M)^{\otimes r}$ for the ordinary $r$-fold tensor power and put
$\cS(M)^{\otimes0}=\mathbb C$.

\subsection{Complex \texorpdfstring{$K$}{K}-theory and power operations}

Let $K(M)$ denote complex topological $K$-theory.  An expression such as
$F-G$ means the corresponding virtual bundle.  Its Chern character is
defined additively, $\ch(F-G)=\ch(F)-\ch(G)$, and an elliptic index twisted
by $F-G$ is the difference of the two ordinary twisted indices.

For a complex vector bundle $W$, define the total exterior and symmetric
powers by
\begin{equation}\label{eq:lambda-s-def}
 \Lambda_t(W)=\sum_{j\geq0}t^j\Lambda^jW,
 \qquad
 S_t(W)=\sum_{j\geq0}t^jS^jW.
\end{equation}
They extend multiplicatively to virtual bundles by
\begin{equation}\label{eq:lambda-relations}
 S_t(W)=\Lambda_{-t}(W)^{-1},\qquad
 \Lambda_t(W-Z)=\frac{\Lambda_t(W)}{\Lambda_t(Z)}.
\end{equation}
These are identities in the completed ring $K(M)[[t]]$; see
\cite[Chapter III]{AtiyahKTheory}.  The first coefficients are
\begin{align}
 \Lambda_t(W)&=1+tW+t^2\Lambda^2W+O(t^3),\label{eq:lambda-first}\\
 S_t(W)&=1+tW+t^2S^2W+O(t^3).
 \label{eq:s-first}
\end{align}
The decomposition of the tensor square gives
\begin{equation}\label{eq:tensor-decomp}
                   W\otimes W=S^2W+\Lambda^2W
\end{equation}
in $K(M)$.  Since the power operations are defined on the whole
$\lambda$-ring $K(M)$, all three expressions in \eqref{eq:tensor-decomp}
also make sense for a virtual class.  The identity itself follows for
virtual classes by comparing the coefficient of $t^2$ in
$S_t(W)\Lambda_{-t}(W)=1$.

For a rank-$r$ bundle $W$, its reduced class is
\[
                        \widetilde W=W-\mathbb C^r.
\]
The central reduced class of this paper is
\begin{equation}\label{eq:E-def}
              E=\widetilde{T_{\mathbb C}M}
               =T_{\mathbb C}M-\mathbb C^{4k}.
\end{equation}
It has virtual rank zero.  Reduction eliminates the constant Fourier
contributions that would otherwise obscure the modular coefficients.

\subsection{Characteristic forms}

For a real Euclidean bundle $W$ with connection $\nabla^W$ and curvature
$R^W=(\nabla^W)^2$, define the Hirzebruch forms
\begin{align}
 \Ahat(W,\nabla^W)
 &=\det{}^{1/2}\!\left(
      \frac{\frac{\sqrt{-1}}{4\pi}R^W}
           {\sinh(\frac{\sqrt{-1}}{4\pi}R^W)}\right),
 \label{eq:Ahat-form}\\
 \Lhat(W,\nabla^W)
 &=\det{}^{1/2}\!\left(
      \frac{\frac{\sqrt{-1}}{2\pi}R^W}
           {\tanh(\frac{\sqrt{-1}}{4\pi}R^W)}\right).
 \label{eq:Lhat-form}
\end{align}
This is the normalization used in \cite{LiuWang,ChenHan}.  For a Hermitian
bundle $F$ with unitary connection and curvature $R^F$, the Chern character
form is
\begin{equation}\label{eq:chern-form}
              \ch(F,\nabla^F)=\operatorname{tr}
              \exp\!\left(\frac{\sqrt{-1}}{2\pi}R^F\right).
\end{equation}
When a characteristic form is paired with $[M]$, we suppress the
connections because the result depends only on the corresponding
cohomology classes.

If the formal Chern roots of $T_{\mathbb C}M$ are
$\{\pm x_j\}_{j=1}^{2k}$, the splitting principle gives
\begin{align}
 \Ahat(TM)&=\prod_{j=1}^{2k}\frac{x_j/2}{\sinh(x_j/2)},\label{eq:A-roots}\\
 \Lhat(TM)&=\prod_{j=1}^{2k}\frac{x_j}{\tanh(x_j/2)},\label{eq:L-roots}\\
 \ch(\cS(M))&=\prod_{j=1}^{2k}
                  \left(e^{x_j/2}+e^{-x_j/2}\right).
 \label{eq:spinor-roots}
\end{align}

\begin{lemma}\label{lem:L-A-spinor}
On a spin manifold, with the normalization above,
\begin{equation}\label{eq:L-A-spinor}
              \Lhat(TM)=\Ahat(TM)\ch(\cS(M)).
\end{equation}
\end{lemma}

\begin{proof}
It suffices by the splitting principle to compare one factor for each
$x_j$.  Multiplying the factors in \eqref{eq:A-roots} and
\eqref{eq:spinor-roots} gives
\[
 \frac{x_j/2}{\sinh(x_j/2)}
 \left(e^{x_j/2}+e^{-x_j/2}\right)
 =\frac{x_j\cosh(x_j/2)}{\sinh(x_j/2)}
 =\frac{x_j}{\tanh(x_j/2)},
\]
which is the factor in \eqref{eq:L-roots}.  Taking the product proves the
identity.
\end{proof}

The normalization of $\Lhat$ must be distinguished from the usual
total Hirzebruch class $L_{\mathrm{std}}(TM)=\prod_jx_j/\tanh x_j$.
In dimension $4k$,
\[
 \Lhat(TM)^{(4j)}=2^{2k-2j}L_{\mathrm{std}}(TM)^{(4j)}.
\]
Their top-degree terms agree, but their lower-degree terms do not.
Consequently $\Lhat$ cannot be replaced by $L_{\mathrm{std}}$ in a
twisted index pairing.  On a spin manifold, the Clifford-module
identification of complex differential forms with
$\cS(M)\otimes\cS(M)^*$, with signature grading on the first factor,
gives the index density $\Ahat\ch(\cS(M))\ch(F)$, as used below.

\subsection{Twisted Dirac indices and twisted signatures}

Let $F$ be a Hermitian bundle with connection.  The spin Dirac operator
twisted by $F$ has chiral part
\[
 D_F^+:\Gamma(\cS^+(M)\otimes F)
       \longrightarrow\Gamma(\cS^-(M)\otimes F).
\]
The Atiyah--Singer index theorem \cite{AtiyahSingerIII} states that
\begin{equation}\label{eq:AS-dirac}
                \Ind(D_F^+)
       =\langle\Ahat(TM)\ch(F),[M]\rangle.
\end{equation}
The signature operator acts on complex differential forms and is graded by
the usual chirality involution.  Its twist by $F$ has index
\begin{equation}\label{eq:AS-signature}
                \Sig(M,F)
       =\langle\Lhat(TM)\ch(F),[M]\rangle;
\end{equation}
see \cite{Hirzebruch,LawsonMichelsohn}.  For $F=\mathbb C$ this is the
ordinary signature of the cup-product form on middle-dimensional
cohomology.

\begin{definition}\label{def:I-r}
For $r\geq0$ and $F\in K(M)$, define
\begin{equation}\label{eq:I-r}
 \cI_r(F)=\langle\Ahat(TM)
    \ch(\cS(M)^{\otimes r})\ch(F),[M]\rangle.
\end{equation}
\end{definition}

The integrality of \eqref{eq:I-r} follows from \eqref{eq:AS-dirac}; tensoring
the spinor bundle simply changes the twisting bundle.  Lemma
\ref{lem:L-A-spinor} and \eqref{eq:AS-signature} imply, for $r\geq1$,
\begin{equation}\label{eq:I-signature}
 \cI_r(F)=\Sig\!\left(M,\cS(M)^{\otimes(r-1)}\otimes F\right).
\end{equation}
For $b=1$, this is the ordinary $F$-twisted signature; larger values
introduce additional spinor tensor factors.

\begin{remark}\label{rem:virtual-indices}
If $F=F_+-F_-$ is virtual, then
$\cI_r(F)=\cI_r(F_+)-\cI_r(F_-)$.  Consequently every congruence below is an
ordinary congruence of integers, even though reduced and virtual bundles are
used to write it economically.
\end{remark}

\subsection{Theta functions}

Let $\tau\in\mathbb H$ and $q=e^{2\pi\sqrt{-1}\tau}$.  We use the product
normalizations
\begin{align}
 \theta(v,\tau)
 &=2q^{1/8}\sin(\pi v)
   \prod_{n=1}^{\infty}(1-q^n)
   (1-e^{2\pi\sqrt{-1}v}q^n)
   (1-e^{-2\pi\sqrt{-1}v}q^n),\label{eq:theta0}\\
 \theta_1(v,\tau)
 &=2q^{1/8}\cos(\pi v)
   \prod_{n=1}^{\infty}(1-q^n)
   (1+e^{2\pi\sqrt{-1}v}q^n)
   (1+e^{-2\pi\sqrt{-1}v}q^n),\label{eq:theta1}\\
 \theta_2(v,\tau)
 &=\prod_{n=1}^{\infty}(1-q^n)
   (1-e^{2\pi\sqrt{-1}v}q^{n-1/2})
   (1-e^{-2\pi\sqrt{-1}v}q^{n-1/2}),\label{eq:theta2}\\
 \theta_3(v,\tau)
 &=\prod_{n=1}^{\infty}(1-q^n)
   (1+e^{2\pi\sqrt{-1}v}q^{n-1/2})
   (1+e^{-2\pi\sqrt{-1}v}q^{n-1/2}).\label{eq:theta3}
\end{align}
These conventions agree with \cite[Section 2]{LiuWang} and
\cite[Section 1]{ChenHan}.  Define
\begin{align}
 \delta_1(\tau)&=\frac18\bigl(\theta_2(0,\tau)^4+
                                      \theta_3(0,\tau)^4\bigr),
 &\varepsilon_1(\tau)&=\frac1{16}\theta_2(0,\tau)^4
                                      \theta_3(0,\tau)^4,
 \label{eq:delta1-def}\\
 \delta_2(\tau)&=-\frac18\bigl(\theta_1(0,\tau)^4+
                                      \theta_3(0,\tau)^4\bigr),
 &\varepsilon_2(\tau)&=\frac1{16}\theta_1(0,\tau)^4
                                      \theta_3(0,\tau)^4.
 \label{eq:delta2-def}
\end{align}
Direct multiplication of the products gives
\begin{align}
 \delta_1&=\frac14+6q+6q^2+O(q^3),
 &\varepsilon_1&=\frac1{16}-q+7q^2+O(q^3),
 \label{eq:delta1-exp}\\
 \delta_2&=-\frac18-3q^{1/2}-3q-12q^{3/2}+O(q^2),
 &\varepsilon_2&=q^{1/2}+8q+28q^{3/2}+O(q^2).
 \label{eq:delta2-exp}
\end{align}

Let
\[
\begin{aligned}
 \Gamma_0(2)&=\left\{
 \begin{pmatrix}a&b\\c&d\end{pmatrix}\in SL_2(\mathbb Z):
 c\equiv0\pmod2\right\},\\
 \Gamma^0(2)&=\left\{
 \begin{pmatrix}a&b\\c&d\end{pmatrix}\in SL_2(\mathbb Z):
 b\equiv0\pmod2\right\}.
\end{aligned}
\]
The following standard structure statement is used only to obtain the finite
expansion in \eqref{eq:R2-basis}.  It is proved, with the same normalizations,
in \cite[Lemma 2.2]{LiuCMP}; see also \cite[Section 1.2]{ChenHan}.

\begin{proposition}[Structure of the level-two rings]\label{prop:mod-rings}
The modular forms with real Fourier coefficients, trivial character, and
holomorphic expansions at every cusp for $\Gamma_0(2)$ are generated by $\delta_1$ and
$\varepsilon_1$, of weights $2$ and $4$, respectively.  The real modular
forms for $\Gamma^0(2)$ are generated by $\delta_2$ and $\varepsilon_2$, with
the same weights.  Under $S:\tau\mapsto-1/\tau$,
\begin{equation}\label{eq:delta-transform}
 \delta_2(-1/\tau)=\tau^2\delta_1(\tau),
 \qquad
 \varepsilon_2(-1/\tau)=\tau^4\varepsilon_1(\tau).
\end{equation}
\end{proposition}

\begin{proof}
The theta transformation laws interchange $\theta_1$ and $\theta_2$
under $S$, leave $\theta_3$ fixed, and multiply each theta constant by
$(-\sqrt{-1}\tau)^{1/2}$.  Substitution into
\eqref{eq:delta1-def}--\eqref{eq:delta2-def} proves
\eqref{eq:delta-transform}.  Under $\tau\mapsto\tau+1$, the constants
$\theta_2$ and $\theta_3$ are interchanged.  Combining these laws with
the generators $T,ST^2S$ of $\Gamma_0(2)$, up to the central element,
gives the asserted weights and groups.  The product expansions are
holomorphic at infinity; the $S$ relations verify holomorphy at the
other cusp.  Conjugation by $S$ gives the corresponding statements for
$\Gamma^0(2)$.

For completeness, fix weight $2k$ and put $m=\lfloor k/2\rfloor$.
The $m+1$ forms
\[
 f_r=(8\delta_2)^{k-2r}\varepsilon_2^r,
 \qquad 0\leq r\leq m,
\]
have leading terms $(-1)^kt^r$, where $t=q^{1/2}$ is the local
parameter at the infinity cusp of $\Gamma^0(2)$.  They are therefore
linearly independent.  The valence formula for this index-three group
bounds the order at that cusp of a nonzero holomorphic weight-$2k$
form by $(2k)\cdot3/12=k/2$.  Given any such form, successive
subtraction of multiples of $f_0,\ldots,f_m$ kills its coefficients
through $t^m$.  The remainder has order at least $m+1>k/2$ and must
vanish.  Thus these forms span the weight-$2k$ space.  The elimination
uses real coefficients when the initial Fourier coefficients are real.
Applying the $S$ relation proves the ring statement on the other group.
\end{proof}

\section{Theta bundles, modularity, and Fourier coefficients}
\label{sec:theta-bundles}

Let $V\to M$ be a real oriented Euclidean bundle of even rank $2l$, and
write $V_{\mathbb C}$ for its complexification. Liu and Wang introduced the
following elements of $K(M)[[q^{1/2}]]$ in
\cite[equations (3.1)--(3.2)]{LiuWang}:
\begin{align}
 \Theta_1(T_{\mathbb C}M,V_{\mathbb C},a,b)
 :={}&\bigotimes_{n=1}^{\infty}S_{q^n}(\widetilde{T_{\mathbb C}M})
 \otimes\left(\bigotimes_{m=1}^{\infty}
      \Lambda_{q^m}(\widetilde{V_{\mathbb C}})\right)^{\!a}
 \notag\\
 &\otimes\left(\bigotimes_{r=1}^{\infty}
      \Lambda_{q^{r-1/2}}(\widetilde{V_{\mathbb C}})\right)^{\!b}
 \otimes\left(\bigotimes_{s=1}^{\infty}
      \Lambda_{-q^{s-1/2}}(\widetilde{V_{\mathbb C}})\right)^{\!b},
 \label{eq:Theta1-general}\\
 \Theta_2(T_{\mathbb C}M,V_{\mathbb C},a,b)
 :={}&\bigotimes_{n=1}^{\infty}S_{q^n}(\widetilde{T_{\mathbb C}M})
 \otimes\left(\bigotimes_{m=1}^{\infty}
      \Lambda_{q^m}(\widetilde{V_{\mathbb C}})\right)^{\!b}
 \notag\\
 &\otimes\left(\bigotimes_{r=1}^{\infty}
      \Lambda_{q^{r-1/2}}(\widetilde{V_{\mathbb C}})\right)^{\!b}
 \otimes\left(\bigotimes_{s=1}^{\infty}
      \Lambda_{-q^{s-1/2}}(\widetilde{V_{\mathbb C}})\right)^{\!a},
 \label{eq:Theta2-general}
\end{align}
where $a,b$ are integers in the formal power-series construction.
We restrict to $a,b\geq0$ so that the spinor powers entering the index
pairings are actual tensor powers.  The reduced bundles and Fourier
coefficients themselves are integral virtual $K$-classes.

Specializing to $V=TM$ gives $\widetilde{V_{\mathbb C}}=E$ and $l=2k$.
Define the scalar series
\begin{align}
 R_1^{a,b}(\tau)
 &=\left\langle\Ahat(TM)
      \ch(\cS(M)^{\otimes a})
      \ch(\Theta_1(\T,\T,a,b)),[M]\right\rangle,
 \label{eq:R1-def}\\
 R_2^{a,b}(\tau)
 &=\left\langle\Ahat(TM)
      \ch(\cS(M)^{\otimes b})
      \ch(\Theta_2(\T,\T,a,b)),[M]\right\rangle.
 \label{eq:R2-def}
\end{align}

The modular input is the following specialization of the Liu--Wang
construction.  We include its theta-function proof to specify both the
normalization and the cohomological use of the string condition.

\begin{theorem}\label{thm:liu-wang-special}
Let $M^{4k}$ be a closed string manifold.  For nonnegative $a,b$, the series
$R_1^{a,b}$ and $R_2^{a,b}$ are modular forms of weight $2k$ over
$\Gamma_0(2)$ and $\Gamma^0(2)$, respectively, and
\begin{equation}\label{eq:LW-transform}
 R_1^{a,b}(-1/\tau)
 =2^{(a-b)2k}\tau^{2k}R_2^{a,b}(\tau).
\end{equation}
\end{theorem}

\begin{proof}
Let $v_j=x_j/(2\pi\sqrt{-1})$ and define
\[
 H(v,\tau)=v\frac{\theta'(0,\tau)}{\theta(v,\tau)},\qquad
 J_\nu(v,\tau)=\frac{\theta_\nu(v,\tau)}{\theta_\nu(0,\tau)}.
\]
The products defining the theta bundles give the total characteristic
classes
\begin{align*}
 \Phi_1(x,\tau)&=2^{2ka}\prod_{j=1}^{2k}
 H(v_j,\tau)J_1(v_j,\tau)^aJ_2(v_j,\tau)^bJ_3(v_j,\tau)^b,\\
 \Phi_2(x,\tau)&=2^{2kb}\prod_{j=1}^{2k}
 H(v_j,\tau)J_1(v_j,\tau)^bJ_2(v_j,\tau)^aJ_3(v_j,\tau)^b.
\end{align*}
Here $R_i^{a,b}=\langle\Phi_i^{(4k)},[M]\rangle$.  Indeed, the
constant term of $H$ is $(x/2)/\sinh(x/2)$, and that of $J_1$ is
$\cosh(x/2)$; the factors $2^{2ka}$ and $2^{2kb}$ account for the
spinor characters.

The Jacobi transformation formulas imply
\begin{align*}
 H(v,-1/\tau)&=e^{-\pi\sqrt{-1}\tau v^2}H(\tau v,\tau),\\
 J_1(v,-1/\tau)&=e^{\pi\sqrt{-1}\tau v^2}J_2(\tau v,\tau),\\
 J_2(v,-1/\tau)&=e^{\pi\sqrt{-1}\tau v^2}J_1(\tau v,\tau),\\
 J_3(v,-1/\tau)&=e^{\pi\sqrt{-1}\tau v^2}J_3(\tau v,\tau).
\end{align*}
Consequently
\begin{equation}\label{eq:root-modular-transform}
 \Phi_1(x,-1/\tau)=2^{2k(a-b)}
 e^{\pi\sqrt{-1}\tau(a+2b-1)\sum_jv_j^2}\Phi_2(\tau x,\tau).
\end{equation}
Since $\sum_jx_j^2=p_1(TM)=0$ in cohomology, the exponential factor
equals $1$ in cohomology.  This assertion does not require the
Pontryagin differential form of the Levi--Civita connection to vanish
pointwise.  If that form is $d\eta$, then for every closed form $\omega$
and every $j\geq1$,
\[
 p_1^j\wedge\omega=d(\eta\wedge p_1^{j-1}\wedge\omega),
 \qquad \int_Mp_1^j\wedge\omega=0.
\]
Thus all nonconstant exponential terms disappear after integration.
Replacing each degree-two root $x_j$ by $\tau x_j$ multiplies the
degree-$4k$ term by $\tau^{2k}$, proving \eqref{eq:LW-transform}.
The inverse $S$ relation follows in the same way.

The normalized theta ratios show that $R_1^{a,b}(\tau+1)=R_1^{a,b}(\tau)$
and $R_2^{a,b}(\tau+2)=R_2^{a,b}(\tau)$.  Together with the two
$S$ relations, these identities give the required transformation laws
on $\Gamma_0(2)$ and $\Gamma^0(2)$.  The theta constants in the
denominators have no zeros on the upper half-plane, and the factor $v$
removes the simple zero of $\theta(v,\tau)$ at $v=0$.  The Taylor
coefficients in the roots are therefore holomorphic.  The bundle
products have no negative Fourier powers at infinity, and the $S$
relations give the same property at the other cusp.  This proves
holomorphy at all cusps.

Equivalently, in the general formula of
\cite[Theorem 3.1 and equation (3.15)]{LiuWang}, the anomaly class is
$p_1(TM)-(a+2b)p_1(V)$ and the scalar is $2^{(a-b)l}$.
The specialization $V=TM$, $l=2k$ gives exactly the formula above.
\end{proof}

We next perform the $K$-theory calculation underlying all later results.

Write $t=q^{1/2}$. For an integer $c\geq0$, the binomial expansion of
\eqref{eq:lambda-first} gives
\begin{align}
 \Lambda_t(E)^c
 &=1+cEt+\left(c\Lambda^2E+\binom c2E\otimes E\right)t^2
   +O(t^3),\label{eq:lambda-power-plus}\\
 \Lambda_{-t}(E)^c
 &=1-cEt+\left(c\Lambda^2E+\binom c2E\otimes E\right)t^2
   +O(t^3).
 \label{eq:lambda-power-minus}
\end{align}
Indeed, a term of degree two is obtained either from the degree-two term in
one factor, giving $c\Lambda^2E$, or from the degree-one terms in two distinct
factors, giving $\binom c2E\otimes E$. Moreover, by \eqref{eq:tensor-decomp}, we get
\begin{align}
 \Lambda_t(E)^b\Lambda_{-t}(E)^b
 &=1+\Bigl(2b\Lambda^2E+2\binom b2E\otimes E-b^2E\otimes E\Bigr)t^2
 +O(t^3)\notag\\
 &=1+b(2\Lambda^2E-E\otimes E)t^2+O(t^3)\notag\\
 &=1+b(\Lambda^2E-S^2E)t^2+O(t^3).
 \label{eq:opposite-product}
\end{align}
These calculations give the following coefficient formulas.
\begin{proposition}\label{prop:theta-coefficients}
For $V=TM$ and $E=\widetilde{T_{\mathbb C}M}$, the theta bundles have
expansions
\begin{align}
 \Theta_1(\T,\T,a,b)
 &=1+\cA_{a,b}q+O(q^2),\label{eq:Theta1-exp}\\
 \Theta_2(\T,\T,a,b)
 &=1+\cB^{(1)}_{a,b}q^{1/2}
       +\cB^{(2)}_{a,b}q+O(q^{3/2}),
 \label{eq:Theta2-exp}
\end{align}
where
\begin{align}
 \cA_{a,b}&=(1+a)E+b(\Lambda^2E-S^2E),\label{eq:A-def}\\
 \cB^{(1)}_{a,b}&=(b-a)E,\label{eq:B1-def}\\
 \cB^{(2)}_{a,b}&=(1+b)E+(a+b)\Lambda^2E
 +\frac{(a-b)^2-(a+b)}2E\otimes E.
 \label{eq:B2-def}
\end{align}
\end{proposition}

\begin{proof}
We calculate the two products separately.  Up to order $q=t^2$, only the
$n=m=1$ integral factors and the $r=s=1$ half-integral factors can
contribute.

For $\Theta_1$, the integral factors are
\begin{align*}
 \prod_{n\geq1}S_{q^n}(E)&=1+Eq+O(q^2),\\
 \left(\prod_{m\geq1}\Lambda_{q^m}(E)\right)^a
 &=1+aEq+O(q^2).
\end{align*}
The two half-integral products have equal exponent $b$.  Their linear terms
cancel, and \eqref{eq:opposite-product} gives
\[
 \left(\prod_{r\geq1}\Lambda_{q^{r-1/2}}(E)\right)^b
 \left(\prod_{s\geq1}\Lambda_{-q^{s-1/2}}(E)\right)^b
 =1+b(\Lambda^2E-S^2E)q+O(q^2).
\]
Multiplying the three displayed expressions proves \eqref{eq:A-def}.
The remainder is $O(q^2)$ because each paired half-integral product
is unchanged under $t\mapsto-t$, and therefore contains only even
powers of $t$.

For $\Theta_2$, the integral factors give
\[
 \left(1+Eq+O(q^2)\right)
 \left(1+bEq+O(q^2)\right)
 =1+(1+b)Eq+O(q^2).
\]
The half-integral factors now have exponents $b$ and $a$ and opposite signs.
Using \eqref{eq:lambda-power-plus}--\eqref{eq:lambda-power-minus}, their
product is
\begin{align*}
 &\left[1+bEt+\left(b\Lambda^2E+\binom b2E^{\otimes2}\right)t^2\right]
 \left[1-aEt+\left(a\Lambda^2E+\binom a2E^{\otimes2}\right)t^2\right]
 +O(t^3)\\
 &\quad=1+(b-a)Et\\
 &\qquad+\left[(a+b)\Lambda^2E+
     \left(\binom a2+\binom b2-ab\right)E^{\otimes2}\right]t^2+O(t^3).
\end{align*}
Finally,
\[
 \binom a2+\binom b2-ab
 =\frac{a^2-a+b^2-b-2ab}{2}
 =\frac{(a-b)^2-(a+b)}2.
\]
Adding the integral contribution proves \eqref{eq:B1-def} and
\eqref{eq:B2-def}.
\end{proof}

Applying Definition~\ref{def:I-r} coefficientwise gives
\begin{align}
 R_1^{a,b}(\tau)
 &=\cI_a(1)+\cI_a(\cA_{a,b})q+O(q^2),
 \label{eq:R1-coeff}\\
 R_2^{a,b}(\tau)
 &=\cI_b(1)+\cI_b(\cB^{(1)}_{a,b})q^{1/2}
   +\cI_b(\cB^{(2)}_{a,b})q+O(q^{3/2}).
 \label{eq:R2-coeff}
\end{align}

Put $T=T_{\mathbb C}M$ and
$d=\dim M$.  From \eqref{eq:lambda-relations},
\begin{align}
 \Lambda^2E&=\Lambda^2T-dT+\binom{d+1}{2}\mathbb C,
 \label{eq:lambda-reduced}\\
 S^2E&=S^2T-dT+\binom d2\mathbb C,
 \label{eq:s-reduced}\\
 E\otimes E&=T\otimes T-2dT+d^2\mathbb C.
 \label{eq:tensor-reduced}
\end{align}
For example, subtracting \eqref{eq:s-reduced} from
\eqref{eq:lambda-reduced} gives
\begin{equation}\label{eq:diff-reduced}
             \Lambda^2E-S^2E=\Lambda^2T-S^2T+d\mathbb C.
\end{equation}
Consequently
\begin{equation}\label{eq:A-unreduced}
 \cA_{a,b}=(1+a)T+b(\Lambda^2T-S^2T)
               +d(b-a-1)\mathbb C.
\end{equation}
If
\begin{equation}\label{eq:c-ab}
               c_{a,b}:=\frac{(a-b)^2-(a+b)}2,
\end{equation}
then
\begin{align}
 \cB^{(2)}_{a,b}
 ={}&(a+b)\Lambda^2T+c_{a,b}T\otimes T\notag\\
 &+\bigl(1+b-d(a+b)-2dc_{a,b}\bigr)T\notag\\
 &+\left[-d(1+b)+(a+b)\binom{d+1}{2}+c_{a,b}d^2\right]\mathbb C.
 \label{eq:B2-unreduced}
\end{align}
The dimension-specific formulas below follow from
\eqref{eq:A-unreduced} and \eqref{eq:B2-unreduced}.

\section{The modular coefficient comparison}\label{sec:comparison}
In this section we prove Theorem~\ref{thm:main-intro}.
\subsection{Triangular coefficients on the
\texorpdfstring{$\Gamma^0(2)$}{Gamma-zero(2)} side}

Since $R_2^{a,b}$ has weight $2k$, Proposition~\ref{prop:mod-rings} gives
numbers $h_r=h_r(M;a,b)$ such that
\begin{equation}\label{eq:R2-basis}
 R_2^{a,b}(\tau)=
 \sum_{r=0}^{\lfloor k/2\rfloor}
 h_r(8\delta_2)^{k-2r}\varepsilon_2^r.
\end{equation}
By \eqref{eq:LW-transform} and \eqref{eq:delta-transform},
\begin{equation}\label{eq:R1-basis}
 R_1^{a,b}(\tau)=2^{(a-b)2k}
 \sum_{r=0}^{\lfloor k/2\rfloor}
 h_r(8\delta_1)^{k-2r}\varepsilon_1^r.
\end{equation}
Set $t=q^{1/2}$, from \eqref{eq:delta2-exp},
\begin{equation}\label{eq:8delta2}
 8\delta_2=-(1+24t+24t^2)+O(t^3),
 \qquad
 \varepsilon_2=t+8t^2+O(t^3).
\end{equation}
For an integer $n\geq0$, the binomial theorem gives
\begin{align}
 (8\delta_2)^n
 &=(-1)^n\left[1+24nt+
 \left(24n+\binom n2 24^2\right)t^2\right]+O(t^3)\notag\\
 &=(-1)^n\left[1+24nt+24n(12n-11)t^2\right]+O(t^3).
 \label{eq:delta2-power}
\end{align}
Only $r=0,1,2$ can contribute through $t^2$. For \(r=0\), the corresponding expansion follows by setting \(n=k\) in the preceding formula. The $r=1$ term
is present only if $k\geq2$; in that case,
\begin{align}
 (8\delta_2)^{k-2}\varepsilon_2
 &=(-1)^k\bigl[1+24(k-2)t+O(t^2)\bigr]
                   \bigl[t+8t^2+O(t^3)\bigr]\notag\\
 &=(-1)^k\bigl[t+(24k-40)t^2\bigr]+O(t^3),
 \label{eq:r1-product}
\end{align}
and the $r=2$ term, present only if $k\geq4$, satisfies
\begin{equation}\label{eq:r2-product}
 (8\delta_2)^{k-4}\varepsilon_2^2=(-1)^k t^2+O(t^3).
\end{equation}

\begin{lemma}\label{lem:h-relations}
For $1\leq k\leq7$, set $h_r=0$ whenever $r>\lfloor k/2\rfloor$.
The first three index coefficients of $R_2^{a,b}$ satisfy
\begin{align}
 \cI_b(1)&=(-1)^k h_0,\label{eq:h0-rel}\\
 \cI_b(\cB^{(1)}_{a,b})
 &=(-1)^k(24kh_0+h_1),\label{eq:h1-rel}\\
 \cI_b(\cB^{(2)}_{a,b})
 &=(-1)^k\bigl(24k(12k-11)h_0+(24k-40)h_1+h_2\bigr).
 \label{eq:h2-rel}
\end{align}
In particular, $h_1$ is omitted for $k=1$, and $h_2$ is omitted for
$k<4$.  Every $h_r$ in the expansion for $k\leq7$ is an integer.
\end{lemma}

\begin{proof}
Set \(t=q^{1/2}\), \(m=\lfloor k/2\rfloor\), and write
\[
F_r(t):=(8\delta_2)^{k-2r}\varepsilon_2^r,
\qquad
R_2^{a,b}(\tau)=\sum_{r=0}^{m}h_rF_r(t).
\]
Using
\[
8\delta_2=-1-24t-24t^2+O(t^3),
\qquad
\varepsilon_2=t+8t^2+O(t^3),
\]
we obtain
\begin{align*}
F_0(t)
&=(-1)^k
  \left[1+24kt+24k(12k-11)t^2\right]+O(t^3),\\
F_1(t)
&=(-1)^k
  \left[t+(24k-40)t^2\right]+O(t^3),
  && k\geq2,\\
F_2(t)
&=(-1)^kt^2+O(t^3),
  && k\geq4.
\end{align*}
Moreover, \(F_r(t)=O(t^3)\) for \(3\leq r\leq m\).
Thus, with \(h_r=0\) for \(r>m\),
\begin{align*}
R_2^{a,b}(\tau)=(-1)^k\Bigl\{
&h_0+(24kh_0+h_1)t\\
&+\bigl[
24k(12k-11)h_0+(24k-40)h_1+h_2
\bigr]t^2
\Bigr\}+O(t^3).
\end{align*}
Comparing this with
\[
R_2^{a,b}(\tau)
=\mathcal I_b(1)
+\mathcal I_b(\mathcal B^{(1)}_{a,b})t
+\mathcal I_b(\mathcal B^{(2)}_{a,b})t^2
+O(t^3)
\]
gives the three asserted identities.

For integrality, put
\[
N_j:=[t^j]R_2^{a,b}(\tau)\in\mathbb Z,
\qquad
u_{r,j}:=[t^j]F_r(t),
\]
where \(N_j\) is an integer by its index interpretation.
The expansions
\begin{align*}
8\delta_2
&=-1-24t-24t^2-96t^3+O(t^4),\\
\varepsilon_2
&=t+8t^2+28t^3+O(t^4),
\end{align*}
together with \(m\leq3\), imply
\[
u_{r,j}\in\mathbb Z,\qquad
u_{r,j}=0\quad(j<r),\qquad
u_{r,r}=(-1)^k,
\qquad 0\leq r,j\leq m.
\]
Comparing the coefficients of \(t^j\), we have
\[
N_j=(-1)^kh_j+\sum_{r=0}^{j-1}h_ru_{r,j}.
\]
Hence
\[
h_0=(-1)^kN_0\in\mathbb Z,
\qquad
h_j=(-1)^k
\left(N_j-\sum_{r=0}^{j-1}h_ru_{r,j}\right),
\quad 1\leq j\leq m.
\]
Induction on \(j\) therefore gives
\(h_j\in\mathbb Z\) for every \(0\leq j\leq m\).
\end{proof}
\subsection{The first coefficient on the
\texorpdfstring{$\Gamma_0(2)$}{Gamma-0(2)} side}

From \eqref{eq:delta1-exp},
\begin{equation}\label{eq:8delta1}
          8\delta_1=2+48q+O(q^2),
 \qquad \varepsilon_1=\frac1{16}-q+O(q^2).
\end{equation}
If $n=k-2r$, then
\begin{align}
 (8\delta_1)^n\varepsilon_1^r
 &=2^n(1+24nq+O(q^2))
   16^{-r}(1-16rq+O(q^2))\notag\\
 &=2^{k-6r}\bigl[1+(24k-64r)q+O(q^2)\bigr].
 \label{eq:delta1-eps-power}
\end{align}
Comparing the coefficient of $q$ in \eqref{eq:R1-coeff} and
\eqref{eq:R1-basis} proves the next lemma.

\begin{lemma}\label{lem:R1-q}
For every $k\geq1$,
\begin{equation}\label{eq:R1-q-formula}
 \cI_a(\cA_{a,b})
 =2^{(a-b)2k}
  \sum_{r=0}^{\lfloor k/2\rfloor}
   2^{k-6r}(24k-64r)h_r.
\end{equation}
\end{lemma}

\subsection{Clearing denominators before taking integer residues}

The identity \eqref{eq:R1-q-formula} has rational coefficients and an
integer left-hand side. We clear its denominators before reducing
modulo $3$.

For fixed $a,b\in\Z_{\geq0}$, put
\begin{equation}\label{eq:integer-UV}
 U_k=2^{2k\max(a-b,0)},\qquad
 V_k=2^{2k\max(b-a,0)}.
\end{equation}
Both are positive integers, and $2^{2k(a-b)}=U_k/V_k$ as an exact
rational equality.  For $1\leq k\leq7$, choose
\begin{equation}\label{eq:integer-D}
 D_k=\begin{cases}
 1,& k=1,2,3,\\
 16,& k=4,5,\\
 256,& k=6,7.
 \end{cases}
\end{equation}
The factors $D_k$ clear the denominators inside the sum in
\eqref{eq:R1-q-formula}.  They have been chosen as powers of $4$ so
that no additional sign occurs modulo $3$.  Define the integers $J_k$
by the following table:
\begin{equation}\label{eq:integer-J-table}
\begin{array}{c|l}
k&J_k\\ \hline
1&48h_0\\
2&192h_0-h_1\\
3&576h_0+h_1\\
4&24576h_0+128h_1-2h_2\\
5&61440h_0+448h_1-h_2\\
6&2359296h_0+20480h_1+64h_2-3h_3\\
7&5505024h_0+53248h_1+320h_2-3h_3.
\end{array}
\end{equation}
Every $h_r$ in this table is an integer by
Lemma~\ref{lem:h-relations}.  Multiplying the exact coefficient sum
in \eqref{eq:R1-q-formula} by $D_k$ gives precisely $J_k$.
For example, at $k=4$ the coefficients
$1536,8,-1/8$ become $24576,128,-2$ after multiplication by $16$;
at $k=6$ the coefficients $9216,80,1/4,-3/256$ become
$2359296,20480,64,-3$ after multiplication by $256$.

Clearing also the outer denominator $V_k$ in
\eqref{eq:R1-q-formula} gives the integer identity
\begin{equation}\label{eq:integer-cleared}
 \boxed{D_kV_k\,\cI_a(\cA_{a,b})=U_kJ_k.}
\end{equation}
Since all exponents in \eqref{eq:integer-UV} are nonnegative and even,
\begin{equation}\label{eq:two-factor-mod3}
 U_k\equiv V_k\equiv D_k\equiv1\pmod3.
\end{equation}
Reducing \eqref{eq:integer-cleared} therefore gives
\begin{equation}\label{eq:integer-residue}
 \cI_a(\cA_{a,b})\equiv J_k\pmod3.
\end{equation}
This also covers $a<b$, since multiplication by $V_k$ clears the
outer denominator before reduction.

\subsection{Proof of the main comparison}

\begin{proof}[Proof of Theorem~\ref{thm:main-intro}]
We combine \eqref{eq:integer-residue} with
Lemma~\ref{lem:h-relations}. All congruences in this proof are modulo $3$.

If $k=1$, \eqref{eq:integer-cleared} reads
\[
 V_1\cI_a(\cA_{a,b})=48U_1h_0.
\]
Its right-hand side is divisible by $3$, and $V_1\equiv1$.
Consequently $\cI_a(\cA_{a,b})\equiv0$, proving
\eqref{eq:main-k1}.

If $k=2$, then $J_2=192h_0-h_1\equiv-h_1$, whereas
\eqref{eq:h1-rel} gives
\[
 \cI_b(\cB^{(1)}_{a,b})=48h_0+h_1\equiv h_1.
\]
Together with \eqref{eq:integer-residue}, this shows that the sum
of the two indices is congruent to zero modulo $3$.
For $k=3$, the corresponding integer expressions are
\[
 J_3=576h_0+h_1\equiv h_1,\qquad
 \cI_b(\cB^{(1)}_{a,b})=-72h_0-h_1\equiv-h_1.
\]
This proves \eqref{eq:main-k23} in both cases.

For $k=4$, the denominator-cleared expression is
\[
 J_4=24576h_0+128h_1-2h_2\equiv2h_1+h_2,
\]
since $24576$ is divisible by $3$, $128\equiv2$, and $-2\equiv1$.
On the other side, \eqref{eq:h2-rel} gives
\[
 \cI_b(\cB^{(2)}_{a,b})=3552h_0+56h_1+h_2
 \equiv2h_1+h_2.
\]
For $k=5$, the two integer expressions give
\begin{align*}
 J_5&=61440h_0+448h_1-h_2\equiv h_1+2h_2,\\
 \cI_b(\cB^{(2)}_{a,b})
 &=-5880h_0-80h_1-h_2\equiv h_1+2h_2.
\end{align*}
The minus sign for odd $k$ is included in the second line.

For $k=6$, use
\begin{align*}
 J_6&=2359296h_0+20480h_1+64h_2-3h_3
       \equiv2h_1+h_2,\\
 \cI_b(\cB^{(2)}_{a,b})
 &=8784h_0+104h_1+h_2\equiv2h_1+h_2.
\end{align*}
In particular, $-3h_3$ disappears because $h_3\in\Z$.
For $k=7$, the same integer comparison is
\begin{align*}
 J_7&=5505024h_0+53248h_1+320h_2-3h_3
       \equiv h_1+2h_2,\\
 \cI_b(\cB^{(2)}_{a,b})
 &=-12264h_0-128h_1-h_2\equiv h_1+2h_2.
\end{align*}
For each $k=4,5,6,7$, equation \eqref{eq:integer-residue} identifies
the residue of the first index with that of $J_k$, so the difference
of the two indices is congruent to zero.  This proves
\eqref{eq:main-k456}.
\end{proof}

\begin{remark}\label{rem:no-mod9}
In dimension $24$, the term $-3h_3$ in $J_6$ vanishes modulo $3$,
but integrality alone does not make it vanish modulo $9$. A comparison
modulo $9$ must retain the factors $D_6,U_6,V_6$ and account for all
coefficients using additional divisibility or geometric relations.
This limits the present argument, but does not rule out stronger
geometric congruences.
\end{remark}

\section{The three parameter families in dimension twenty-four}
\label{sec:families}

We express all index formulas in this section and
Section~\ref{sec:dimensions} using genuine complex vector bundles.
For $d=\dim M$ and $T=T_{\mathbb C}M$, put
\begin{equation}\label{eq:g56-indices}
\begin{aligned}
j_r&=\cI_r(\mathbb C),& t_r&=\cI_r(T),&
\ell_r&=\cI_r(\Lambda^2T),\\
s_r&=\cI_r(S^2T),& z_r&=\cI_r(T\otimes T)=s_r+\ell_r.
\end{aligned}
\end{equation}
Thus each symbol is the index of $D$ twisted by a genuine bundle,
including the factor $\cS(M)^{\otimes r}$ in $\cI_r$.
These are the indices denoted by $J_r,U_r,V_r,W_r,Z_r$ in
Corollary~\ref{cor:genuine-uniform}; the lower-case notation here
distinguishes them from the clearing factors $U_k,V_k$ and the scalar
expression $J_k$ in Section~\ref{sec:comparison}.
For $r\geq1$,
\[
\cI_r(F)=\Sig(M,\cS(M)^{\otimes(r-1)}\otimes F),
\qquad F=\mathbb C,T,\Lambda^2T,S^2T,T\otimes T.
\]

To keep the exact coefficient comparisons concise, define the following
integer combinations of genuine-bundle indices:
\begin{equation}\label{eq:g56-combinations}
\begin{aligned}
P^{(d)}_{a,b}&=(a+1)t_a+b\ell_a-bs_a+d(b-a-1)j_a,\\
Q^{(d)}_{a,b}&=(b-a)(t_b-dj_b),\\
H^{(d)}_{a,b}&=(a+b)\ell_b+c_{a,b}z_b
                 +\alpha_{a,b}^{(d)}t_b+\beta_{a,b}^{(d)}j_b,
\end{aligned}
\end{equation}
where
\begin{equation}\label{eq:g56-coefficients}
\begin{aligned}
c_{a,b}&=\binom a2+\binom b2-ab,\\
\alpha_{a,b}^{(d)}&=1+b-d(a+b)-2dc_{a,b},\\
\beta_{a,b}^{(d)}&=-d(1+b)+(a+b)\binom{d+1}{2}+d^2c_{a,b}.
\end{aligned}
\end{equation}
The exact substitutions in \eqref{eq:lambda-reduced},
\eqref{eq:s-reduced}, and \eqref{eq:tensor-reduced} give
\[
[q]R_1^{a,b}=P^{(d)}_{a,b},\qquad
[q^{1/2}]R_2^{a,b}=Q^{(d)}_{a,b},\qquad
[q]R_2^{a,b}=H^{(d)}_{a,b}.
\]
The symbols $P,Q,H$ denote integer combinations of genuine-bundle indices.

Throughout the rest of this section $d=24$. Since
$24,\binom{25}{2},24^2$ are divisible by $3$,
Theorem~\ref{thm:main-intro} becomes
\begin{equation}\label{eq:24-master}
(a+1)t_a+b(\ell_a-s_a)
\equiv (b+1)t_b+(a+b)\ell_b+c_{a,b}z_b\pmod3.
\end{equation}
We treat $(0,0)$, $(a,0)$, and $(0,b)$ separately.

For the parameter choices considered below, the corresponding
coefficient classes in dimension $24$ are listed in the following
table. These exact identities follow from
\eqref{eq:lambda-reduced}--\eqref{eq:tensor-reduced},
with $T=T_{\mathbb C}M$ and $E=T-\mathbb C^{24}$.

\begin{equation}\label{eq:24-parameter-classes}
\begin{array}{c|c|c|c}
(a,b)&\mathcal A_{a,b}&\mathcal B^{(1)}_{a,b}
     &\mathcal B^{(2)}_{a,b}\\ \hline
(0,0)&T-24&0&T-24\\
(1,0)&2T-48&24-T&\Lambda^2T-23T+276\\
(0,1)&T+\Lambda^2T-S^2T&T-24&
                   \Lambda^2T-22T+252\\
(2,0)&3T-72&48-2T&
                   T^{\otimes2}+2\Lambda^2T-95T+1152\\
(0,2)&T+2\Lambda^2T-2S^2T+24&2T-48&
                   T^{\otimes2}+2\Lambda^2T-93T+1104
\end{array}
\end{equation}

Here each integer denotes the corresponding multiple of the trivial
complex line bundle, and all equalities are in $K(M)$.

\subsection{The weight-twelve coefficient calculation}
Since $k=6$, the modular expansion is
\begin{equation}\label{eq:dim24-R2}
R_2^{a,b}=h_0(8\delta_2)^6+h_1(8\delta_2)^4\varepsilon_2
+h_2(8\delta_2)^2\varepsilon_2^2+h_3\varepsilon_2^3.
\end{equation}
The last term starts at $q^{3/2}$. Lemma~\ref{lem:h-relations} gives
the exact integer relations
\begin{align}
j_b&=h_0,\label{eq:dim24-h0}\\
Q^{(24)}_{a,b}&=144h_0+h_1,\label{eq:dim24-h1}\\
H^{(24)}_{a,b}&=8784h_0+104h_1+h_2.\label{eq:dim24-h2}
\end{align}
The transformed expression is
\begin{equation}\label{eq:dim24-R1}
\begin{split}
R_1^{a,b}=2^{12(a-b)}\{&h_0(8\delta_1)^6+h_1(8\delta_1)^4\varepsilon_1\\
&+h_2(8\delta_1)^2\varepsilon_1^2+h_3\varepsilon_1^3\}.
\end{split}
\end{equation}
Its $q$ coefficient is
\begin{equation}\label{eq:dim24-A-exact}
P^{(24)}_{a,b}=2^{12(a-b)}
\left(9216h_0+80h_1+\frac14h_2-\frac3{256}h_3\right).
\end{equation}
Multiplying by $256V_6$ gives an equality of integers:
\begin{equation}\label{eq:dim24-A-integer}
256V_6P^{(24)}_{a,b}
=U_6(2359296h_0+20480h_1+64h_2-3h_3).
\end{equation}
Since $256,U_6,V_6\equiv1\pmod3$,
\[
P^{(24)}_{a,b}\equiv2h_1+h_2
\equiv H^{(24)}_{a,b}\pmod3.
\]
Expanding $P^{(24)}_{a,b},H^{(24)}_{a,b}$ by
\eqref{eq:g56-combinations} proves \eqref{eq:24-master}.

\subsection{The case \texorpdfstring{$(a,b)=(0,0)$}{(a,b)=(0,0)}}
For $(a,b)=(0,0)$, the exact coefficient combinations are
\[
P^{(24)}_{0,0}=H^{(24)}_{0,0}=t_0-24j_0,
\qquad Q^{(24)}_{0,0}=0.
\]
Thus \eqref{eq:24-master} is the tautology $t_0\equiv t_0$.
A nontrivial identity comes from the next coefficient of the Witten genus.
The following is the identity of Chen and Han \cite[Lemma 2.1]{ChenHan}.

\begin{proposition}[Chen--Han]\label{prop:CH-Witten}
For a closed $24$-dimensional string manifold,
\begin{equation}\label{eq:CH-exact}
s_0+t_0=196884j_0.
\end{equation}
Consequently,
\begin{equation}\label{eq:CH-mod3}
\Ind(D\otimes S^2T)+\Ind(D\otimes T)\equiv0\pmod3.
\end{equation}
\end{proposition}
\begin{proof}
Write the Witten bundle using the genuine bundle $T$:
\begin{equation}\label{eq:witten-bundle}
\Theta_W(T)=\left(\bigotimes_{n\geq1}S_{q^n}(T)\right)
                 \prod_{n\geq1}(1-q^n)^{24}.
\end{equation}
Its low-order coefficients follow from
\begin{align*}
\bigotimes_{n\geq1}S_{q^n}(T)
 &=1+Tq+(S^2T+T)q^2+O(q^3),\\
\prod_{n\geq1}(1-q^n)^{24}
 &=1-24q+252q^2+O(q^3).
\end{align*}
Hence, with $\mathcal W_M=\langle\Ahat(TM)\ch(\Theta_W(T)),[M]\rangle$,
\begin{equation}\label{eq:Witten-exp-24}
\mathcal W_M=j_0+(t_0-24j_0)q+(s_0-23t_0+252j_0)q^2+O(q^3).
\end{equation}
The string condition makes this a weight-twelve modular form for
$SL_2(\mathbb Z)$ \cite{WittenElliptic,LiuCMP}. Thus
\begin{equation}\label{eq:weight12-basis}
\mathcal W_M=mE_4^3+n\Delta_{\mathrm{mod}}.
\end{equation}
Using
\begin{align}
E_4^3&=1+720q+179280q^2+O(q^3),\label{eq:E4cube}\\
\Delta_{\mathrm{mod}}&=q-24q^2+O(q^3),\label{eq:Delta-mod}
\end{align}
comparison gives
\begin{align}
j_0&=m,\label{eq:CH-elim0}\\
t_0-24j_0&=720m+n,\label{eq:CH-elim1}\\
s_0-23t_0+252j_0&=179280m-24n.\label{eq:CH-elim2}
\end{align}
In particular $m,n\in\mathbb Z$ and $n=t_0-744m$. Substitution yields
\[
s_0-23t_0+252m=197136m-24t_0,
\qquad s_0+t_0=196884m.
\]
Since $3\mid196884$, \eqref{eq:CH-mod3} follows.
\end{proof}

\subsection{The family \texorpdfstring{$(a,0)$}{(a,0)}}
Set $b=0$ and $C_a=\binom a2$. The exact genuine-index combinations are
\begin{align}
P^{(24)}_{a,0}&=(a+1)(t_a-24j_a),\label{eq:a0-A}\\
Q^{(24)}_{a,0}&=-a(t_0-24j_0),\label{eq:a0-B1}\\
H^{(24)}_{a,0}&=a\ell_0+C_az_0+(1-24a-48C_a)t_0
                    +(-24+300a+576C_a)j_0.\label{eq:a0-B2}
\end{align}
\begin{corollary}\label{cor:a0-24}
For every $a\geq0$ and every closed $24$-dimensional string manifold,
\begin{equation}\label{eq:a0-24}
(a+1)t_a\equiv t_0+a\ell_0+\binom a2z_0\pmod3.
\end{equation}
\end{corollary}
\begin{proof}
Substitute $b=0$ in \eqref{eq:24-master}. Every omitted rank-correction
coefficient in \eqref{eq:a0-A}--\eqref{eq:a0-B2} is divisible by $3$.
\end{proof}
The coefficient $a+1$ vanishes modulo $3$ precisely when
$a\equiv2\pmod3$; $t_a$ itself may have additional divisibility
for other parameters.

\begin{corollary}\label{cor:a2-24}
For every closed $24$-dimensional string manifold,
\begin{align}
t_0+2\ell_0+z_0&\equiv0\pmod3,\label{eq:a2-first}\\
s_0+t_0&\equiv0\pmod3.\label{eq:a2-second}
\end{align}
\end{corollary}
\begin{proof}
Setting $a=2$ in \eqref{eq:a0-24} gives
\[
0\equiv3t_2\equiv t_0+2\ell_0+z_0
=t_0+s_0+3\ell_0\pmod3.
\]
Both conclusions follow because $z_0=s_0+\ell_0$ and $\ell_0\in\mathbb Z$.
\end{proof}
The level-two comparison therefore recovers the mod $3$ consequence
of Proposition~\ref{prop:CH-Witten}. More generally,
$a\equiv2\pmod3$ implies $\binom a2\equiv1\pmod3$ and produces the same
right-hand residue $t_0+s_0$.

\subsection{The family \texorpdfstring{$(0,b)$}{(0,b)}}
Set $a=0$ and $C_b=\binom b2$. Then
\begin{align}
P^{(24)}_{0,b}&=t_0+b(\ell_0-s_0)+24(b-1)j_0,\label{eq:0b-A}\\
Q^{(24)}_{0,b}&=b(t_b-24j_b),\label{eq:0b-B1}\\
H^{(24)}_{0,b}&=b\ell_b+C_bz_b+(1-23b-48C_b)t_b
                       +(-24+276b+576C_b)j_b.\label{eq:0b-B2}
\end{align}
\begin{corollary}\label{cor:0b-24}
For every $b\geq0$ and every closed $24$-dimensional string manifold,
\begin{equation}\label{eq:0b-24}
t_0+b(\ell_0-s_0)
\equiv(b+1)t_b+b\ell_b+\binom b2z_b\pmod3.
\end{equation}
\end{corollary}
\begin{proof}
This is \eqref{eq:24-master} with $a=0$; all five types of indices
in \eqref{eq:g56-indices} are integers.
\end{proof}
If $b=3m+2$, then
\[
\binom b2-1=\frac{9m(m+1)}2\in3\mathbb Z,
\qquad b+1\equiv0,\quad b\equiv-1\pmod3.
\]
It follows that
\begin{equation}\label{eq:B2-b2-class}
H^{(24)}_{0,b}\equiv-\ell_b+z_b=s_b\pmod3.
\end{equation}
The least such parameter is $b=2$. For comparison, $b=1$ gives an
ordinary twisted signature, since $\cI_1(F)=\Sig(M,F)$; $b=2$ gives
a signature with one additional spinor factor.

\begin{theorem}\label{thm:b2-24}
For every closed $24$-dimensional string manifold,
\begin{equation}\label{eq:b2-24-final}
s_2+\ell_0\equiv0\pmod3.
\end{equation}
Equivalently, with genuine twisting bundles,
\begin{equation}\label{eq:b2-24-signature}
\Sig(M,\cS(M)\otimes S^2T)
+\Ind(D\otimes\Lambda^2T)\equiv0\pmod3.
\end{equation}
\end{theorem}
\begin{proof}
At $b=2$, Corollary~\ref{cor:0b-24} and $z_2=s_2+\ell_2$ give
\[
t_0+2\ell_0-2s_0
\equiv3t_2+2\ell_2+z_2
=s_2+3(t_2+\ell_2)\equiv s_2\pmod3.
\]
Corollary~\ref{cor:a2-24} gives $t_0+s_0\equiv0$, so
\[
t_0+2\ell_0-2s_0
\equiv t_0-\ell_0+s_0\equiv-\ell_0\pmod3.
\]
Thus $s_2+\ell_0\equiv0$. The signature formulation follows from
\eqref{eq:I-signature}.
\end{proof}
For a fixed index subscript $r$, the coefficient reductions above
depend only on the parameters modulo $3$. This does not imply
periodicity of the full pairings $\mathcal I_a$ and $\mathcal I_b$,
since their twisting bundles contain the actual tensor powers
$\mathcal S(M)^{\otimes a}$ and $\mathcal S(M)^{\otimes b}$.
\subsection{Recovery of the Chen--Han comparison}
The choices $(a,b)=(0,1)$ and $(1,0)$ in \eqref{eq:24-master} give
\begin{align}
t_0+\ell_0-s_0&\equiv2t_1+\ell_1
                         \equiv\ell_1-t_1\pmod3,\label{eq:recover-01}\\
t_1&\equiv-t_0-\ell_0\pmod3.\label{eq:recover-10}
\end{align}
Together with $s_0+t_0\equiv0$, these imply
\[
\ell_1\equiv t_1+t_0+\ell_0-s_0
\equiv-s_0\equiv t_0\pmod3,
\]
or
\[
\Sig(M,\Lambda^2T)\equiv\Ind(D\otimes T)\pmod3.
\]
This recovers \cite[Theorem 0.1]{ChenHan}. The additional divisibility
of the Rarita--Schwinger index by $3$ is a known geometric result;
it is not inferred solely from this two-index comparison.

For completeness, the same coefficients recover the exact identity of
\cite[Lemma 2.3]{ChenHan}. At $(a,b)=(1,0)$, put $m=j_0$; then
\[
h_0=m,\qquad h_1=-t_0-120m,\qquad
h_2=\ell_0+81t_0+3972m.
\]
The constant and first coefficients of $R_1^{1,0}$ give
\begin{align*}
j_1&=2^{12}\left(64h_0+h_1+\frac{h_2}{64}+\frac{h_3}{4096}\right),\\
2t_1-48j_1&=2^{12}
\left(9216h_0+80h_1+\frac{h_2}{4}-\frac{3h_3}{256}\right).
\end{align*}
Eliminating $h_3$ yields
\[
2t_1=2^{12}(12288h_0+128h_1+h_2).
\]
Substitution therefore gives the genuine-index identity
\begin{equation}\label{eq:recover-exact-10}
t_1=2^{11}(\ell_0-47t_0+900j_0).
\end{equation}
Indeed, $12288-15360+3972=900$ and $-128+81=-47$.
For this parameter pair, $1-a-2b=0$, so this last exact identity
also holds for closed $24$-dimensional spin manifolds without the
string hypothesis.

\section{The remaining dimensions}\label{sec:dimensions}
We now use the genuine-bundle indices \eqref{eq:g56-indices} and the
exact integer combinations \eqref{eq:g56-combinations} with
$d=\dim M$. The clearing factors $U_k,V_k$ retain their definitions
from Section~\ref{sec:comparison}. All congruences in the first seven
subsections are modulo $3$. Together with Section~\ref{sec:families},
these computations prove Theorem~\ref{thm:selected-intro}.

\subsection{Dimension four}
Here $d=4$ and $k=1$. The modular expansions are
\begin{align*}
R_2^{a,b}&=-h_0-24h_0q^{1/2}-24h_0q+O(q^{3/2}),\\
R_1^{a,b}&=2^{2(a-b)}h_0(2+48q+O(q^2)).
\end{align*}
Consequently,
\[
V_1P^{(4)}_{a,b}=48U_1h_0,
\]
and $V_1\equiv1\pmod3$ gives
\begin{equation}\label{eq:dim4-general}
(a+1)t_a+b(\ell_a-s_a)+4(b-a-1)j_a\equiv0\pmod3.
\end{equation}
In fact the degree-four component of each relevant index density is a
multiple of $p_1(TM)$. The string condition therefore gives
$j_r=t_r=\ell_r=s_r=z_r=0$ for every $r\geq0$.
In particular the three families satisfy
\begin{align}
(0,0):&\quad t_0=0,\label{eq:dim4-00}\\
(a,0):&\quad(a+1)(t_a-4j_a)=0,\label{eq:dim4-a0}\\
(0,b):&\quad t_0+b(\ell_0-s_0)+4(b-1)j_0=0.\label{eq:dim4-0b}
\end{align}
These equalities provide the dimension-four specializations without
using a virtual twisting bundle in an index.

\subsection{Dimension eight}
Here $d=8$ in \eqref{eq:g56-combinations}.

Now $k=2$, and
\begin{equation}\label{eq:dim8-R2}
 R_2^{a,b}=h_0(8\delta_2)^2+h_1\varepsilon_2.
\end{equation}
Using \eqref{eq:8delta2},
\begin{align*}
 (8\delta_2)^2&=1+48q^{1/2}+624q+O(q^{3/2}),\\
 \varepsilon_2&=q^{1/2}+8q+O(q^{3/2}).
\end{align*}
Therefore
\begin{align}
 j_b&=h_0,\label{eq:dim8-h0}\\
 Q^{(8)}_{a,b}&=48h_0+h_1,
 \label{eq:dim8-h1}\\
 H^{(8)}_{a,b}&=624h_0+8h_1.
 \label{eq:dim8-h2}
\end{align}
On the other side,
\[
 R_1^{a,b}=2^{4(a-b)}\left[h_0(8\delta_1)^2+h_1\varepsilon_1\right].
\]
After clearing the outer denominator, the $q$ coefficient identity is
\begin{equation}\label{eq:dim8-A}
 V_2P^{(8)}_{a,b}=U_2(192h_0-h_1).
\end{equation}
As $U_2\equiv V_2\equiv1\pmod3$, reduction of this integer equality gives
$P^{(8)}_{a,b}\equiv-h_1\pmod3$.
Meanwhile \eqref{eq:dim8-h1} is congruent to $h_1$.  This reproduces, with
all scalar coefficients visible,
\begin{equation}\label{eq:dim8-master}
 P^{(8)}_{a,b}+Q^{(8)}_{a,b}\equiv0\pmod3.
\end{equation}

The genuine-bundle parameter families are

\begin{equation}\label{eq:dim8-00}
t_0\equiv2j_0\pmod3.
\end{equation}

\begin{equation}\label{eq:dim8-a0}
(a+1)(t_a-8j_a)-a(t_0-8j_0)\equiv0\pmod3.
\end{equation}

\begin{equation}\label{eq:dim8-0b}
t_0+b(\ell_0-s_0)+8(b-1)j_0+b(t_b-8j_b)\equiv0\pmod3.
\end{equation}

At $a=2$, the first term of the $(a,0)$ relation is divisible by
3 and hence $t_0-dj_0\equiv0$. At $b=2$, the $(0,b)$ relation is
\[
t_0+2\ell_0-2s_0+dj_0+2(t_2-dj_2)\equiv0.
\]
Subtracting $t_0-dj_0\equiv0$ and multiplying by $2$ gives
\[
t_2-dj_2+\ell_0-s_0+dj_0\equiv0.
\]
Thus

\begin{equation}\label{eq:dim8-new}
t_2+\ell_0+2j_0\equiv2j_2+s_0\pmod3.
\end{equation}

Here $t_2=\Sig(M,\cS(M)\otimes T)$ and $j_2=\Sig(M,\cS(M))$.

\subsection{Dimension twelve}
Here $d=12$ in \eqref{eq:g56-combinations}.

For $k=3$ the modular basis is
\begin{equation}\label{eq:dim12-R2}
 R_2^{a,b}=h_0(8\delta_2)^3+h_1(8\delta_2)\varepsilon_2.
\end{equation}
Equations \eqref{eq:delta2-power} and \eqref{eq:r1-product} give
\begin{align}
 j_b&=-h_0,\label{eq:dim12-h0}\\
 Q^{(12)}_{a,b}&=-(72h_0+h_1),
 \label{eq:dim12-h1}\\
 H^{(12)}_{a,b}&=-(1800h_0+32h_1).
 \label{eq:dim12-h2}
\end{align}
Indeed, $24k(12k-11)=72\cdot25=1800$ and $24k-40=32$.

The transformed expression is
\[
 R_1^{a,b}=2^{6(a-b)}\left[
 h_0(8\delta_1)^3+h_1(8\delta_1)\varepsilon_1\right].
\]
Clearing the outer denominator in its $q$ coefficient gives
\begin{equation}\label{eq:dim12-A}
 V_3P^{(12)}_{a,b}=U_3(576h_0+h_1).
\end{equation}
All coefficients and factors are integers, so
$U_3\equiv V_3\equiv1\pmod3$ implies
$P^{(12)}_{a,b}\equiv h_1\pmod3$.
Since \eqref{eq:dim12-h1} is $-h_1$ modulo $3$, we obtain
\begin{equation}\label{eq:dim12-master}
 P^{(12)}_{a,b}+Q^{(12)}_{a,b}\equiv0\pmod3.
\end{equation}

The genuine-bundle parameter families are

\begin{equation}\label{eq:dim12-00}
t_0\equiv0\pmod3.
\end{equation}

\begin{equation}\label{eq:dim12-a0}
(a+1)(t_a-12j_a)-a(t_0-12j_0)\equiv0\pmod3.
\end{equation}

\begin{equation}\label{eq:dim12-0b}
t_0+b(\ell_0-s_0)+12(b-1)j_0+b(t_b-12j_b)\equiv0\pmod3.
\end{equation}

At $a=2$, the first term of the $(a,0)$ relation is divisible by
3 and hence $t_0-dj_0\equiv0$. At $b=2$, the $(0,b)$ relation is
\[
t_0+2\ell_0-2s_0+dj_0+2(t_2-dj_2)\equiv0.
\]
Subtracting $t_0-dj_0\equiv0$ and multiplying by $2$ gives
\[
t_2-dj_2+\ell_0-s_0+dj_0\equiv0.
\]
Thus

\begin{equation}\label{eq:dim12-new}
t_2+\ell_0\equiv s_0\pmod3.
\end{equation}

Here $t_2=\Sig(M,\cS(M)\otimes T)$ and $j_2=\Sig(M,\cS(M))$.

\subsection{Dimension sixteen}
Here $d=16$ in \eqref{eq:g56-combinations}.

When $k=4$, the basis acquires a third term:
\begin{equation}\label{eq:dim16-R2}
 R_2^{a,b}=h_0(8\delta_2)^4
 +h_1(8\delta_2)^2\varepsilon_2+h_2\varepsilon_2^2.
\end{equation}
The first three coefficients are
\begin{align}
 j_b&=h_0,\label{eq:dim16-h0}\\
 Q^{(16)}_{a,b}&=96h_0+h_1,
 \label{eq:dim16-h1}\\
 H^{(16)}_{a,b}&=3552h_0+56h_1+h_2.
 \label{eq:dim16-h2}
\end{align}
Here $24\cdot4(48-11)=3552$.  On the $R_1$ side, multiply the
exact coefficient identity by $16V_4$ to obtain
\begin{equation}\label{eq:dim16-A}
 16V_4P^{(16)}_{a,b}
 =U_4(24576h_0+128h_1-2h_2).
\end{equation}
Reducing only this integer equality, using
$16\equiv U_4\equiv V_4\equiv1\pmod3$, gives
\[
 P^{(16)}_{a,b}\equiv2h_1+h_2\pmod3.
\]
Equation \eqref{eq:dim16-h2} has the same reduction, and hence
\begin{equation}\label{eq:dim16-master}
 P^{(16)}_{a,b}-H^{(16)}_{a,b}\equiv0\pmod3.
\end{equation}

For the two parameter families, put $C_a=\binom a2$ and $C_b=\binom b2$. Then

\begin{equation}\label{eq:dim16-a0}
\begin{split}
(a+1)(t_a-16j_a)
&\equiv a\ell_0+C_az_0+(1-16a-32C_a)t_0\\
&\quad+(-16+136a+256C_a)j_0\pmod3.
\end{split}
\end{equation}

\begin{equation}\label{eq:dim16-0b}
\begin{split}
t_0+b(\ell_0-s_0)+16(b-1)j_0
&\equiv b\ell_b+C_bz_b+(1-15b-32C_b)t_b\\
&\quad+(-16+120b+256C_b)j_b\pmod3.
\end{split}
\end{equation}

Taking $(a,b)=(2,0)$ and using $z_0=s_0+\ell_0$ gives
\[
0\equiv H^{(d)}_{2,0}
\equiv s_0+(1-d)t_0+\frac{d(d-3)}2j_0\pmod3.
\]

\begin{equation}\label{eq:dim16-a2}
s_0\equiv j_0\pmod3.
\end{equation}

For $(a,b)=(0,2)$, reduction of the exact combinations gives
\begin{align*}
P^{(d)}_{0,2}&\equiv t_0-\ell_0+s_0+dj_0,\\
H^{(d)}_{0,2}&\equiv s_2-dt_2+\binom d2j_2.
\end{align*}
Using the preceding Dirac congruence and $P^{(d)}_{0,2}\equiv H^{(d)}_{0,2}$,
\[
s_2-dt_2+\binom d2j_2
+\ell_0-dt_0+\binom{d+1}{2}j_0\equiv0\pmod3.
\]
Substituting the dimension and reducing modulo $3$ gives

\begin{equation}\label{eq:dim16-new}
s_2+\ell_0+j_0\equiv t_2+t_0\pmod3.
\end{equation}

Here $s_2=\Sig(M,\cS(M)\otimes S^2T)$ and $t_2=\Sig(M,\cS(M)\otimes T)$.

\subsection{Dimension twenty}
Here $d=20$ in \eqref{eq:g56-combinations}.

For $k=5$ the modular basis is
\begin{equation}\label{eq:dim20-R2}
 R_2^{a,b}=h_0(8\delta_2)^5
 +h_1(8\delta_2)^3\varepsilon_2
 +h_2(8\delta_2)\varepsilon_2^2.
\end{equation}
The triangular relations are
\begin{align}
 j_b&=-h_0,\label{eq:dim20-h0}\\
 Q^{(20)}_{a,b}&=-(120h_0+h_1),
 \label{eq:dim20-h1}\\
 H^{(20)}_{a,b}&=-(5880h_0+80h_1+h_2),
 \label{eq:dim20-h2}
\end{align}
because $24\cdot5(60-11)=5880$.  On the transformed side,
clearing all denominators gives
\begin{equation}\label{eq:dim20-A}
 16V_5P^{(20)}_{a,b}
 =U_5(61440h_0+448h_1-h_2).
\end{equation}
Since $16\equiv U_5\equiv V_5\equiv1\pmod3$, this integer identity implies
\[
 P^{(20)}_{a,b}\equiv h_1+2h_2\pmod3.
\]
The reduction of \eqref{eq:dim20-h2} is also $h_1+2h_2$, since the leading
minus sign changes $80h_1+h_2$ into $h_1+2h_2$ modulo $3$.  Therefore
\begin{equation}\label{eq:dim20-master}
 P^{(20)}_{a,b}-H^{(20)}_{a,b}\equiv0\pmod3.
\end{equation}

For the two parameter families, put $C_a=\binom a2$ and $C_b=\binom b2$. Then

\begin{equation}\label{eq:dim20-a0}
\begin{split}
(a+1)(t_a-20j_a)
&\equiv a\ell_0+C_az_0+(1-20a-40C_a)t_0\\
&\quad+(-20+210a+400C_a)j_0\pmod3.
\end{split}
\end{equation}

\begin{equation}\label{eq:dim20-0b}
\begin{split}
t_0+b(\ell_0-s_0)+20(b-1)j_0
&\equiv b\ell_b+C_bz_b+(1-19b-40C_b)t_b\\
&\quad+(-20+190b+400C_b)j_b\pmod3.
\end{split}
\end{equation}

Taking $(a,b)=(2,0)$ and using $z_0=s_0+\ell_0$ gives
\[
0\equiv H^{(d)}_{2,0}
\equiv s_0+(1-d)t_0+\frac{d(d-3)}2j_0\pmod3.
\]

\begin{equation}\label{eq:dim20-a2}
s_0\equiv t_0+j_0\pmod3.
\end{equation}

For $(a,b)=(0,2)$, reduction of the exact combinations gives
\begin{align*}
P^{(d)}_{0,2}&\equiv t_0-\ell_0+s_0+dj_0,\\
H^{(d)}_{0,2}&\equiv s_2-dt_2+\binom d2j_2.
\end{align*}
Using the preceding Dirac congruence and $P^{(d)}_{0,2}\equiv H^{(d)}_{0,2}$,
\[
s_2-dt_2+\binom d2j_2
+\ell_0-dt_0+\binom{d+1}{2}j_0\equiv0\pmod3.
\]
Substituting the dimension and reducing modulo $3$ gives

\begin{equation}\label{eq:dim20-new}
s_2+\ell_0+j_2\equiv2(t_2+t_0)\pmod3.
\end{equation}

Here $s_2=\Sig(M,\cS(M)\otimes S^2T)$ and $t_2=\Sig(M,\cS(M)\otimes T)$.

\subsection{Dimension twenty-eight}
Here $d=28$ in \eqref{eq:g56-combinations}.

Here $k=7$ and
\begin{equation}\label{eq:dim28-R2}
\begin{split}
 R_2^{a,b}={}&h_0(8\delta_2)^7+h_1(8\delta_2)^5\varepsilon_2\\
 &+h_2(8\delta_2)^3\varepsilon_2^2+h_3(8\delta_2)\varepsilon_2^3.
\end{split}
\end{equation}
The coefficient identities are
\begin{align}
 j_b&=-h_0,\label{eq:dim28-h0}\\
 Q^{(28)}_{a,b}&=-168h_0-h_1,\label{eq:dim28-h1}\\
 H^{(28)}_{a,b}&=-12264h_0-128h_1-h_2.\label{eq:dim28-h2}
\end{align}
For the transformed $q$ coefficient, multiply by $256V_7$:
\begin{equation}\label{eq:dim28-A}
 256V_7P^{(28)}_{a,b}
 =U_7(5505024h_0+53248h_1+320h_2-3h_3).
\end{equation}
This is an integer identity, and $256,U_7,V_7$ are all congruent to
$1$ modulo $3$.  It follows that both
$P^{(28)}_{a,b}$ and $H^{(28)}_{a,b}$ have residue
$h_1+2h_2$ modulo $3$.

Taking $(a,b)=(2,0)$ and using $z_0=s_0+\ell_0$ gives
\[
0\equiv H^{(d)}_{2,0}
\equiv s_0+(1-d)t_0+\frac{d(d-3)}2j_0\pmod3.
\]

\begin{equation}\label{eq:dim28-a2}
s_0\equiv j_0\pmod3.
\end{equation}

For $(a,b)=(0,2)$, reduction of the exact combinations gives
\begin{align*}
P^{(d)}_{0,2}&\equiv t_0-\ell_0+s_0+dj_0,\\
H^{(d)}_{0,2}&\equiv s_2-dt_2+\binom d2j_2.
\end{align*}
Using the preceding Dirac congruence and $P^{(d)}_{0,2}\equiv H^{(d)}_{0,2}$,
\[
s_2-dt_2+\binom d2j_2
+\ell_0-dt_0+\binom{d+1}{2}j_0\equiv0\pmod3.
\]
Substituting the dimension and reducing modulo $3$ gives

\begin{equation}\label{eq:dim28-new}
s_2+\ell_0+j_0\equiv t_2+t_0\pmod3.
\end{equation}

Here $s_2=\Sig(M,\cS(M)\otimes S^2T)$ and $t_2=\Sig(M,\cS(M)\otimes T)$.

Together with the dimension-$24$ formulas $s_0+t_0\equiv0$ and
$s_2+\ell_0\equiv0$, these complete the proof of
Theorem~\ref{thm:selected-intro}.

\subsection{The next coefficient in dimension thirty-two}
The exact formula of Lemma~\ref{lem:R1-q} still applies when $k=8$,
but involves $h_4$.  Put $D_8=1024$ and define the integer expression
\[
 J_8=50331648h_0+524288h_1+4096h_2-h_4.
\]
Clearing the denominators in that formula gives
\[
 1024V_8P^{(32)}_{a,b}=U_8J_8.
\]
To justify integrality of $h_4$, the triangular induction needs one
more Fourier order.  The theta products give
\begin{align*}
8\delta_2&=-1-24t-24t^2-96t^3-24t^4+O(t^5),\\
\varepsilon_2&=t+8t^2+28t^3+64t^4+O(t^5).
\end{align*}
Thus all coefficients needed in this additional step are integral,
and $h_4\in\Z$ follows as before.  Reducing the cleared integer
identity, since $1024\equiv U_8\equiv V_8\equiv1\pmod3$, gives
\[
 P^{(32)}_{a,b}\equiv2h_1+h_2+2h_4\pmod3.
\]
On the other side the integer coefficient formula is
\[
 H^{(32)}_{a,b}=16320h_0+152h_1+h_2
 \equiv2h_1+h_2\pmod3.
\]
The difference therefore has residue $2h_4$.  The same comparison
does not automatically extend to dimension $32$: it requires
divisibility of $h_4$ by $3$ or another relation.  This observation
does not assert that such a condition fails for any particular
string manifold.

\subsection{Concise congruences modulo higher powers of three}
\label{sec:higher-three}

\begin{proof}[Proof of Theorem~\ref{thm:higher-three}]
Let $m=J_0=j_0\in\mathbb Z$, so $U_0=t_0$ and $W_0=s_0$.
In dimension $d$, the genuine-bundle product expression for the Witten bundle is
\[
\Theta_W(T)=\left(\bigotimes_{n\geq1}S_{q^n}(T)\right)
                  \prod_{n\geq1}(1-q^n)^d.
\]
The two factors have expansions
\begin{align*}
\bigotimes_{n\geq1}S_{q^n}(T)&=1+Tq+(S^2T+T)q^2+O(q^3),\\
\prod_{n\geq1}(1-q^n)^d&=1-dq+\frac{d(d-3)}2q^2+O(q^3).
\end{align*}
Their product and the index theorem give
\begin{equation}\label{eq:witten-two-coeff-general}
\begin{split}
 \mathcal W_M(\tau)
 &:={\big\langle}\Ahat(TM)\ch(\Theta_W(T)),[M]{\big\rangle}\\
 &=m+(U_0-dm)q
   +\left(W_0+(1-d)U_0+\frac{d(d-3)}2m\right)q^2+O(q^3).
\end{split}
\end{equation}
Under the string hypothesis this is a holomorphic modular form of weight
$d/2$ for $SL_2(\mathbb Z)$.  The classical structure theorem
$M_*(SL_2(\mathbb Z);\mathbb C)=\mathbb C[E_4,E_6]$
(see also \cite[Theorem 1.1]{ChenHan}) gives
\[
 M_8=\mathbb C E_4^2,\qquad M_{10}=\mathbb C E_4E_6,
 \qquad M_{14}=\mathbb C E_4^2E_6.
\]
Here the subscript denotes modular weight.  Each indicated generator
has constant coefficient one.  Consequently, in dimensions $16,20,28$,
respectively, $\mathcal W_M=mE_4^2$, $mE_4E_6$, and $mE_4^2E_6$.
The standard expansions
\[
 E_4=1+240q+2160q^2+O(q^3),\qquad
 E_6=1-504q-16632q^2+O(q^3)
\]
yield
\begin{equation}\label{eq:higher-three-series}
\begin{array}{c|r|r}
 f&[q]f&[q^2]f\\ \hline
 E_4^2&480&61920\\
 E_4E_6&-264&-135432\\
 E_4^2E_6&-24&-196632
\end{array}.
\end{equation}
If $A=[q]f$ and $B=[q^2]f$, comparison with
\eqref{eq:witten-two-coeff-general} gives the exact identities
\[
 U_0=(d+A)m,\qquad
 W_0=\left(B+(d-1)(d+A)-\frac{d(d-3)}2\right)m.
\]
Thus, before taking any residues,
\begin{equation}\label{eq:higher-three-exact}
\begin{array}{c|r|r}
 d&U_0&W_0\\ \hline
 16&496m&69256m\\
 20&-244m&-140238m\\
 28&4m&-196874m
\end{array}.
\end{equation}
In particular,
\begin{align*}
 d=16:&\quad W_0-J_0=69255m=729\cdot95m,\\
 d=20:&\quad W_0=-140238m=-27\cdot5194m,\qquad
                    U_0+J_0=-243m,\\
 d=28:&\quad W_0-J_0=-196875m=-9\cdot21875m.
\end{align*}
These integer identities prove three rows of
\eqref{eq:higher-three-table} and the additional dimension-$20$ assertion.

For $d=24$, the space has two generators, so the one-dimensional
argument must be replaced by the coefficient elimination already given
in Proposition~\ref{prop:CH-Witten}.  Its exact Chen--Han identity
\eqref{eq:CH-exact} says
\[
 W_0+U_0=196884m=27\cdot7292m.
\]
This proves the remaining row as an arithmetic consequence of the
Chen--Han identity.
\end{proof}

\begin{remark}\label{rem:higher-three-scope}
The refinements in \eqref{eq:higher-three-table} use additional
relations among the selected Dirac indices, so they are compatible
with the limitation in Remark~\ref{rem:no-mod9}. They do not strengthen
the spinor-twisted column or the general parameter comparison, and
the moduli are not claimed to be optimal on string bordism classes.
The corresponding non-string examples in
Example~\ref{ex:genuine-selected} also violate these refinements,
since they already fail modulo $3$.
\end{remark}

\section{Examples}\label{sec:examples}
This section presents counterexamples showing that the congruences
in Theorems~\ref{thm:main-intro} and~\ref{thm:selected-intro}
do not extend to all closed spin manifolds without the string
hypothesis. The first example tests the uniform comparison in
dimensions $8$ and $16$, while the second treats both selected
congruences in every dimension of Theorem~\ref{thm:selected-intro}.
The corresponding Dirac counterexamples also rule out such an
extension of the stronger congruences in
Theorem~\ref{thm:higher-three}.
\begin{bundleexample}
\label{ex:genuine-uniform}
As in \cite[Section 3]{ChenHan}, let $H=\mathbb{HP}^2$ with its standard
orientation.  It is spin since $H^2(H;\mathbb Z/2)=0$, and
\[
 H^*(H;\mathbb Z)=\mathbb Z[u]/(u^3),\qquad
 \langle u^2,[H]\rangle=1,\qquad p_1(TH)=2u,\quad p_2(TH)=7u^2.
\]
For an eight-dimensional manifold, the characteristic expansions needed
here are
\begin{align*}
 \Ahat&=1-\frac{p_1}{24}+\frac{7p_1^2-4p_2}{5760},\\
 \ch(T)&=8+p_1+\frac{p_1^2-2p_2}{12},\\
 \ch(\Lambda^2T)&=28+6p_1+\frac{p_1^2}{2},\\
 \ch(S^2T)&=36+10p_1+\frac{11p_1^2}{6}-\frac{8p_2}{3},\\
 \Ahat\ch(\cS)^{2}
  &=256\left(1+\frac{5p_1}{24}
                  +\frac{67p_1^2+116p_2}{5760}\right),
\end{align*}
where only terms of degree at most eight are displayed.  The square-power
identities used here are
$\ch(S^2T)=(\ch(T)^2+\ch(\psi^2T))/2$ and
$\ch(\Lambda^2T)=(\ch(T)^2-\ch(\psi^2T))/2$, where
$\psi^2$ doubles each Chern root.  Integration gives
\[
\begin{array}{c|rrrr}
 &J_r(H)&U_r(H)&V_r(H)&W_r(H)\\ \hline
 r=0&0&-1&1&-13\\
 r=2&48&384&3136&960
\end{array}.
\]
For $a=b=0$, \eqref{eq:genuine-uniform-23} would assert
$U_0-8J_0\equiv0\pmod3$, whereas on $H$ this integer is $-1$.
This disproves an extension of that congruence to all spin manifolds.

For the higher-dimensional comparison, take $M=H\times H$ and
$(a,b)=(2,0)$.  Multiplicativity of the index and the decomposition of
squares of a direct sum give
\[
 (J_0,U_0,V_0,W_0)=(0,0,1,1),\qquad
 (J_2,U_2)=(2304,36864),\qquad Z_0=V_0+W_0=2.
\]
Here $d=16$, $c=1$, $\alpha=-63$, and $\beta=512$.
The left side of \eqref{eq:genuine-uniform-high} is
$3U_2-48J_2=0$, while its right side is
$2V_0+Z_0-63U_0+512J_0=4\not\equiv0\pmod3$.
Moreover, $p_1(TM)=2u_1+2u_2\ne0$.
These examples show that the $p_1$-vanishing hypothesis cannot simply be
dropped from the uniform conclusion.  They do not assert failure for
every non-string manifold or every parameter choice.
\end{bundleexample}

\begin{bundleexample}
\label{ex:genuine-selected}
We follow the product-manifold method of \cite[Section 3]{ChenHan}.
Retain $H=\mathbb{HP}^2$ and the index values from
Example~\ref{ex:genuine-uniform}.  Let $K$ be a K3 surface with its complex
orientation.  Thus $K$ is spin,
$\langle p_1(TK),[K]\rangle=-48$, and $\Ahat(K)=2$.
Let $B$ be the eight-dimensional spin Bott manifold used in
\cite[Section 3]{ChenHan}, with $\Ahat(B)=1$ and $\Sig(B)=0$.
Consider the following closed spin manifolds:
\[
\begin{array}{c|l}
 d&M_d\\ \hline
 8&H\\
 12&K\times H\\
 16&H\times H\\
 20&K\times H\times H\\
 24&B\times H\times H\\
 28&K\times B\times H\times H
\end{array}.
\]
Each has $p_1(TM_d)\ne0$: restriction to any displayed $H$ factor gives
$2u\ne0$.  In dimension $24$ this is exactly the product used by Chen
and Han to show the necessity of the string hypothesis for their
congruence; here we test the different congruences in
Theorem~\ref{thm:selected-intro}.

For clarity, define the integer residuals corresponding to the two
congruences before coefficient reduction:
\[
 R_D(M)=
 \begin{cases}
 U_0-dJ_0,&d=8,12,\\
 W_0+(1-d)U_0+\dfrac{d(d-3)}2J_0,&d=16,20,24,28,
 \end{cases}
\]
\[
 R_S(M)=
 \begin{cases}
 U_2-dJ_2+V_0-W_0+dJ_0,&d=8,12,\\
 W_2-dU_2+\dbinom d2J_2+V_0-dU_0+\dbinom{d+1}{2}J_0,
       &d=16,20,24,28.
 \end{cases}
\]
Both would have to vanish modulo $3$ if the corresponding row of
\eqref{eq:genuine-selected-table} held without the $p_1$ condition.
Instead, direct calculation gives
\[
\begin{array}{c|rr}
 d&R_D(M_d)\bmod3&R_S(M_d)\bmod3\\ \hline
 8&2&2\\
 12&1&1\\
 16&1&1\\
 20&2&2\\
 24&1&1\\
 28&2&2
\end{array}.
\]
We give the computations to verify both columns.  The full spinor bundle
and the $\Ahat$-class are multiplicative on these even-dimensional
products.  For two factors $X,Y$, with pulled-back bundles understood,
\begin{align*}
 J_r(X\times Y)&=J_r(X)J_r(Y),\\
 U_r(X\times Y)&=U_r(X)J_r(Y)+J_r(X)U_r(Y),\\
 V_r(X\times Y)&=V_r(X)J_r(Y)+J_r(X)V_r(Y)+U_r(X)U_r(Y),\\
 W_r(X\times Y)&=W_r(X)J_r(Y)+J_r(X)W_r(Y)+U_r(X)U_r(Y).
\end{align*}
These are the same direct-sum and multiplicativity identities as those
used in \cite[Lemma 3.1]{ChenHan}.
For $H$ the values already computed yield
\[
 R_D(H)=-1,\qquad
 R_S(H)=384-8\cdot48+1-(-13)=14.
\]
For $K$, the degree-four expansions
\begin{align*}
 \Ahat&=1-\frac{p_1}{24},&
 \ch(\cS)&=4\left(1+\frac{p_1}{8}\right),&
 \ch(T)&=4+p_1,\\
 \ch(\Lambda^2T)&=6+2p_1,& \ch(S^2T)&=10+6p_1
\end{align*}
give
\[
 (J_0,U_0,V_0,W_0)(K)=(2,-40,-84,-268),\qquad
 (J_2,U_2)(K)=(-160,-1408).
\]
Consequently, on $K\times H$,
\[
 (J_0,U_0,V_0,W_0)=(0,-2,42,14),\qquad
 (J_2,U_2)=(-7680,-129024),
\]
and hence
\[
 R_D(K\times H)=-2,\qquad
 R_S(K\times H)=-129024+12\cdot7680+42-14=-36836.
\]
Both integers are $1$ modulo $3$.

For the four higher dimensions, write $M_d=H\times H\times X$, where
$X$ is respectively a point, $K$, $B$, or $K\times B$.
The product formulas give
\[
 J_0(M_d)=U_0(M_d)=0,\qquad
 V_0(M_d)=W_0(M_d)=\Ahat(X).
\]
Moreover, $J_2(H)=48$, $U_2(H)=384$, and $W_2(H)=960$ are all
divisible by $3$.  The product formulas therefore imply
\[
 J_2(M_d)\equiv U_2(M_d)\equiv W_2(M_d)\equiv0\pmod3.
\]
It follows that
\[
 R_D(M_d)\equiv R_S(M_d)\equiv\Ahat(X)\pmod3,
 \qquad \Ahat(X)=1,2,1,2,
\]
which proves the remaining four rows.  In particular, on the reference
product $B\times H\times H$,
\[
 \Ind(D\otimes S^2T)+\Ind(D\otimes T)=1,
\]
\[
 \Sig(M,\cS(M)\otimes S^2T)+\Ind(D\otimes\Lambda^2T)
       \equiv1\pmod3.
\]
Thus both concise formulas fail in each listed dimension after the
$p_1$-vanishing hypothesis is removed.  The conclusion is the failure of
a general extension to spin manifolds, not a claim that nonzero $p_1$
forces failure on every individual manifold.
\end{bundleexample}

\section*{Acknowledgements}
 The authors were supported by Science and Technology Development Plan Project of Jilin Province, China: No.20260102245JC, NSFC. No.12401059, the Fundamental
Research Funds for the Central Universities N2405015 and Liaoning Province Science and Technology Plan Joint Project 2023-BSBA-118. The authors thank the referee for his (or her) careful reading and helpful comments.

\vskip 1 true cm

\section*{Data availability}

No data were collected for this article.

\section*{Conflict of interest}

The authors have no relevant financial or non-financial interests to disclose.


\begin{thebibliography}{99}

\bibitem{AtiyahKTheory}
M. F. Atiyah,
\emph{K-Theory},
W. A. Benjamin, New York--Amsterdam, 1967.

\bibitem{AtiyahSingerIII}
M. F. Atiyah and I. M. Singer,
The index of elliptic operators. III,
\emph{Ann. of Math. (2)} \textbf{87} (1968), 546--604.

\bibitem{ChenHan}
Q. Chen and F. Han,
Mod $3$ congruence and twisted signature of $24$ dimensional string
manifolds,
\emph{Trans. Amer. Math. Soc.} \textbf{367} (2015), 2959--2977.

\bibitem{Hirzebruch}
F. Hirzebruch,
\emph{Topological Methods in Algebraic Geometry},
3rd ed., Grundlehren Math. Wiss. 131, Springer, Berlin, 1966.

\bibitem{LandweberStong}
P. S. Landweber and R. E. Stong,
Circle actions on spin manifolds and characteristic numbers,
\emph{Topology} \textbf{27} (1988), 145--161.

\bibitem{LawsonMichelsohn}
H. B. Lawson, Jr. and M.-L. Michelsohn,
\emph{Spin Geometry},
Princeton Math. Ser. 38, Princeton University Press, Princeton, NJ, 1989.

\bibitem{LiuCMP}
K. Liu,
Modular invariance and characteristic numbers,
\emph{Comm. Math. Phys.} \textbf{174} (1995), 29--42.

\bibitem{LiuTopology}
K. Liu,
On elliptic genera and theta-functions,
\emph{Topology} \textbf{35} (1996), 617--640.

\bibitem{LiuWang}
K. Liu and Y. Wang,
A note on modular forms and generalized anomaly cancellation formulas,
\emph{Sci. China Math.} \textbf{56} (2013), 55--65.

\bibitem{Ochanine}
S. Ochanine,
Signature modulo $16$, invariants de Kervaire g\'en\'eralis\'es et nombres
caract\'eristiques dans la $K$-th\'eorie r\'eelle,
\emph{M\'em. Soc. Math. France (N.S.)} No. 5 (1981), 1--142.

\bibitem{Rokhlin}
V. A. Rokhlin,
New results in the theory of four-dimensional manifolds,
\emph{Doklady Akad. Nauk SSSR (N.S.)} \textbf{84} (1952), 221--224.

\bibitem{StolzTeichner}
S. Stolz and P. Teichner,
What is an elliptic object?,
in \emph{Topology, Geometry and Quantum Field Theory},
London Math. Soc. Lecture Note Ser. 308, Cambridge University Press,
Cambridge, 2004, 247--343.

\bibitem{WittenElliptic}
E. Witten,
Elliptic genera and quantum field theory,
\emph{Comm. Math. Phys.} \textbf{109} (1987), 525--536.

\end{thebibliography}
\end{document}